\documentclass[11pt, reqno]{amsart}
\usepackage{indentfirst, amssymb, amsmath, amsthm, mathrsfs, setspace, indentfirst, enumerate,  mathrsfs, amsmath, amsthm, graphicx, MnSymbol}
\usepackage[bookmarksnumbered, colorlinks, plainpages,]{hyperref}
\usepackage{mathrsfs}
\usepackage{tikz}
\usepackage{float}
\usepackage{booktabs}
\usepackage{graphicx}
\usetikzlibrary{
	positioning,
	calc,
	shadows,
	shadows.blur
}
\usetikzlibrary{arrows.meta}
\usepackage{pgfplots}
\usepackage{placeins}
\usepgfplotslibrary{colormaps}
\usepackage{tabularx, colortbl, xcolor}
\usepackage{array, enumitem}
\pgfplotsset{compat=1.18}
\usepgfplotslibrary{fillbetween}
\usetikzlibrary{3d}
\usepackage{tikz-3dplot}
\usepackage{smartdiagram}
\usesmartdiagramlibrary{additions}
\usepackage[most]{tcolorbox}
\usetikzlibrary{matrix,positioning}
\tdplotsetmaincoords{70}{120}
\newtheorem*{cor A}{Corollary A}
\newtheorem*{cor B}{Corollary B}

\newtheorem{theo}{Theorem}[section]
\newtheorem{lem}{Lemma}[section]

\newtheorem{note}{Note}[section]

\newtheorem{defi}{Definition}[section]

\newcommand{\ol}{\overline}
\newcommand{\be}{\begin{equation}}
	\newcommand{\ee}{\end{equation}}
\newcommand{\beas}{\begin{eqnarray*}}
	\newcommand{\eeas}{\end{eqnarray*}}
\newcommand{\bea}{\begin{eqnarray}}
	\newcommand{\eea}{\end{eqnarray}}

\numberwithin{equation}{section}
\begin{document}
	\title[C\MakeLowercase{oefficient Problems for a......}]{\LARGE C\LARGE\MakeLowercase{oefficient Problems for a \MakeUppercase{M}a--\MakeUppercase{M}inda Convex Class Associated with the Normalized Arcsine Mapping }} %{$\mathcal{C}_{\MakeLowercase{\arcsin}}$}}
	\date{}
	\author[S. P\MakeLowercase{anja}, A. B\MakeLowercase{anerjee}, J. B\MakeLowercase{anerjee} \MakeLowercase{and} S. M\MakeLowercase{ajumder}]{S\MakeLowercase{hantanu} p\MakeLowercase{anja}$^1$$^*$, A\MakeLowercase{bhijit} B\MakeLowercase{anerjee}$^2$, J\MakeLowercase{hilik} B\MakeLowercase{anerjee$^3$ and} S\MakeLowercase{ujoy} M\MakeLowercase{ajumder}$^4$}
	\address{$^{1}$ Department of Mathematics, University of Kalyani, West Bengal 741235, India.}
	\email{panjasantu07@gmail.com, shantanumath26@klyuniv.ac.in}
	
	\address{$^{2}$ Department of Mathematics, University of Kalyani, West Bengal 741235, India.}
	\email{abanerjee\_kal@yahoo.co.in, abanerjeekal@gmail.com}
	
	\address{$^{3}$ Department of Mathematics, University of Kalyani, West Bengal 741235, India.}
	\email{jhilikbanerjee38@gmail.com}

	\address{$^{4}$ Department of Mathematics, Raiganj University, Raiganj, West Bengal-733134, India.}
	\email{sm05math@gmail.com, sjm@raiganjuniversity.ac.in}
	
	\renewcommand{\thefootnote}{}
	\footnote{2020 \emph{Mathematics Subject Classification}: 30C45, 30C50, 30C55}
	\footnote{\emph{Key words and phrases}:Univalent functions, Convex function, Logarithmic coefficients, Coefficient difference, Hankel determinant, Zalcman conjecture, Feketo--Sezeg\"{o} functional.}
	\footnote{*\emph{Corresponding Author}: Shantanu Panja.}
	\renewcommand{\thefootnote}{\arabic{footnote}}
	\setcounter{footnote}{0}

\begin{abstract}
	
	Let $\mathcal{C}_{\arcsin}$ denote the Ma--Minda subclass of convex functions generated by the normalized arcsine mapping $\varphi(z)=1+\frac{2}{\pi}\arcsin z.$
	For this family, we develop a unified coefficient analysis based on subordination techniques, Carath\'eodory functions and sharp estimates for Schwarz functions. As consequences, we derive sharp estimates for the initial Taylor coefficients, logarithmic coefficients and certain differences involving the logarithmic and inverse logarithmic coefficients. We further determine the exact bounds for the second Hankel determinant $H_{2,2}(f)$ together with the Hankel determinants $H_{2,1}(F_f/2)$ and $H_{2,1}(F_{f^{-1}}/2)$ associated with the logarithmic coefficients of a function and its inverse. Moreover, sharp estimates are obtained for the initial generalized Zalcman functional and the generalized Fekete--Szeg\"o functional. In every case, the corresponding extremal functions are identified, showing that all of the obtained inequalities are best possible.
	
\end{abstract}
	
	\thanks{Typeset by \AmS -\LaTeX}
	\maketitle
	
\section{Introduction}

Convexity is one of the fundamental geometric properties in the theory of univalent functions. Besides its intrinsic geometric significance, convexity plays an essential role in the study of coefficient estimates, growth problems, logarithmic coefficients, inverse functions and extremal problems. During the last three decades, differential subordination has provided a unified framework for constructing numerous subclasses of convex functions associated with suitable analytic mappings.

Let $\mathcal{A}$ denote the family of analytic functions in the open unit disk $\Omega=\{z\in\mathbb C:|z|<1\},$ normalized by
\begin{equation}\label{e-1.1}
	f(z)=z+\sum_{n=2}^{\infty}a_n z^n.
\end{equation}
The subclass of univalent functions in $\mathcal A$ is denoted by $\mathcal S$. A function $f\in\mathcal S$ is said to be convex if $f(\Omega)$ is a convex domain. The class of convex functions is characterized analytically by
\[
\mathcal C = \left\{f\in\mathcal S: \Re\!\left(1+\frac{zf''(z)}{f'(z)}\right)>0,\ z\in\Omega \right\}.
\]

We consider $\mathcal{B}_0$ be the class of functions $\omega$ which are analytic in $\Omega$ and satisfies $\omega(0)=0$ and $|\omega(z)|\le 1$. Every function $\omega\in \mathcal{B}_0$ have a Taylor series expansion $\omega(z)=\sum_{n=1}^{\infty}b_nz^n$ for $z\in\Omega$. Also $\mathcal{B}_0$ is called class of Schwarz functions.

\medskip
Let $f$ and $g$ be two analytic functions in $\Omega$. The function $f$ is subordinate to $g$, if there exists a Schwarz function $\omega(z)$ such that $f(z)=g(\omega(z))$ for all $z\in \Omega$ and this is denoted \cite{Lowner_annalen_1923} by $f\prec g$. The concept of subordination used by Ma-Minda \cite{Ma+Minda_1992} in 1992 introduced a broader subclass of $S^*$ and defined by
\[
\mathcal{C}(\psi):=\left\{f\in\mathcal{S}:1+\frac{zf''(z)}{f'(z)}\prec\varphi(z),\;z\in{\Omega} \right\},
\]
where $\varphi$ is an analytic univalent function in $\Omega$ with positive real part, symmetric with respect to the real axis, normalized by $\varphi(0)=1$ and $\varphi'(0)>0$.

\smallskip 
Subordination has become a fundamental tool in geometric function theory for defining and studying various subclasses of analytic and univalent functions. In particular Janowski \cite{Janowski_Anpolon_1970} introduced $\varphi(z)=\frac{1+Az}{1+Bz}$, where $-1\le B< A\le 1$, the classes $S^*(\varphi)$ and $\mathcal{C}(\varphi)$ reduce to the classes $S^{*}[A, B]$ and $\mathcal{C}[A, B]$ respectively. Mendiratta et al. \cite{Mwndiratta et al_BMMS_2015} studied $\varphi(z)=e^z$ with the class $\mathcal{S}^*_e$ and $\mathcal{C}_e$. In 2021 Alimohammadi et al. \cite{DEA_JFT} introduced the convex class $\mathscr{C}(1+\alpha z)$, $0< \alpha\le1$. In 2022 Arora-Kumar \cite{Arora_BKMS_2022} introduced the function $\varphi(z)=1+\operatorname{arcsinh}z$ with the class $S_\rho^*$. Several authors have investigated interesting subclasses of starlike and convex functions \cite{Ronning_AUMC_1991, Kumar_AMP_2021, Kumar_Yadav_IJS_2026, cho, Bano, Alotabi et. al.}.

Despite extensive studies on logarithmic coefficients for starlike classes, the corresponding theory for Ma-Minda convex classes, particularly those associated with transcendental mappings, remains less developed. In particular, no sharp estimates for logarithmic coefficients and their inverse analogues have been established for convex functions defined by the normalized arcsine mapping. This paper fills this gap by providing complete sharp estimates for the class $\mathcal{C}_{\arcsin}$.

The normalized arcsine mapping possesses several geometric features that make it particularly suitable in the Ma--Minda framework. Its image domain is symmetric with respect to the real axis and convex in the horizontal direction (a domain is called convex in the horizontal direction if every horizontal line intersects it in a possibly empty interval), while the mapping itself admits a simple Taylor expansion that allows explicit coefficient calculations. These properties make the associated convex class a natural setting for investigating logarithmic coefficients and their inverse analogues.

\smallskip
Motivated by recent studies on Ma--Minda subclasses associated with special functions, we consider the function $\varphi(z)=1+\frac{2}{\pi}\arcsin z$ for $z\in\Omega$.  

\begin{defi}
	Motivated by the geometric behaviour of the normalized arcsine mapping, we introduce the following Ma--Minda convex subclass.
	\[
	\mathcal{C}_{\arcsin}=\left\{f\in\mathcal S:1+\frac{zf''(z)}{f'(z)}\prec 1+\frac{2}{\pi}\arcsin z, \ \forall z\in\Omega\right\},
	\]
	which represents the subclass of convex functions associated with the leaf-shaped domain shown in Figure~\ref{fig:leaf-domain}.
\end{defi}

\begin{figure}[H]
	\centering
	\includegraphics[width=12cm]{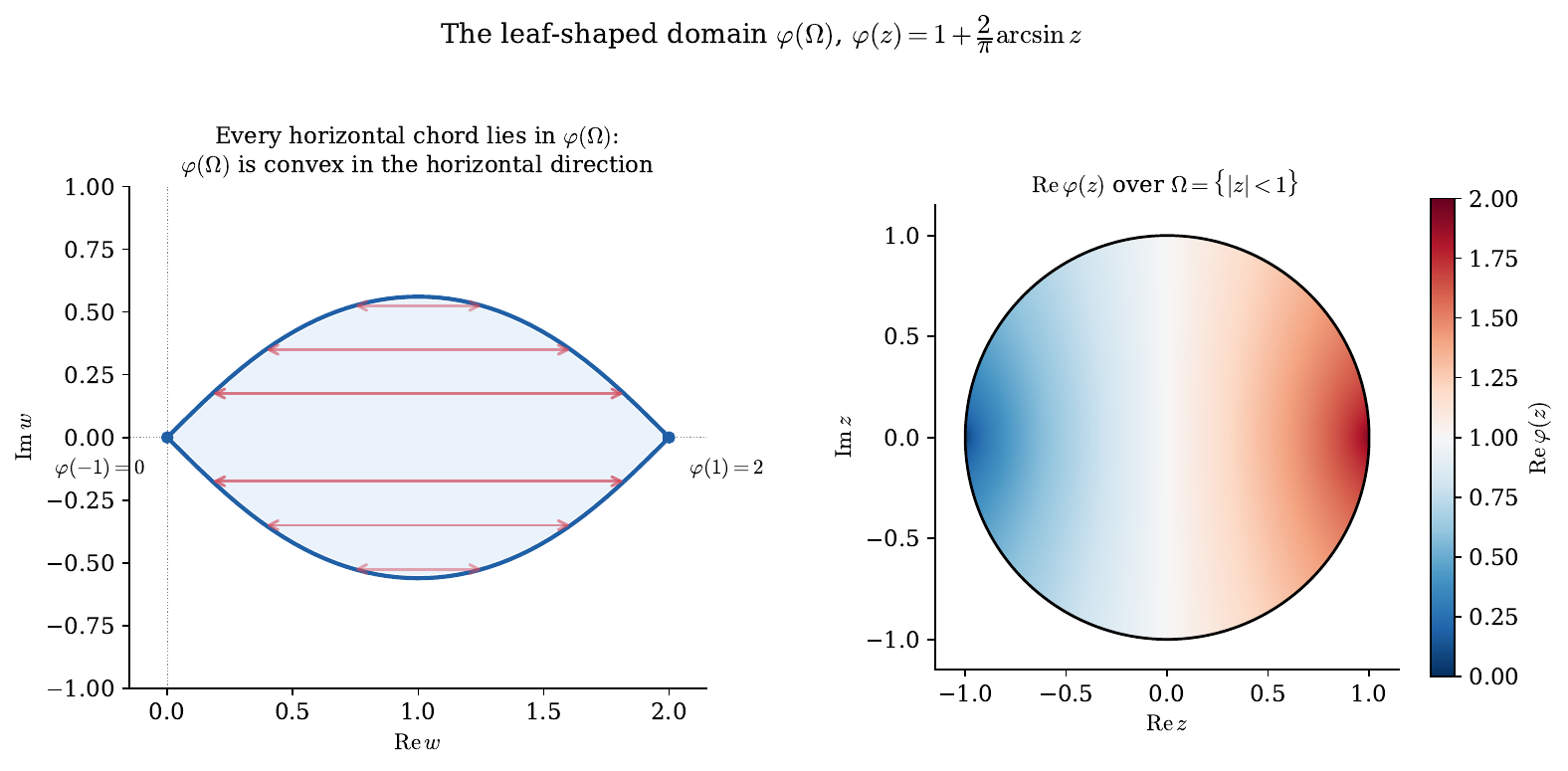}
	\caption{Image of the unit disc $\Omega$ under the normalized mapping $\varphi(z)=1+\dfrac{2}{\pi}\arcsin z$. The boundary $\{\varphi(e^{it}):-\pi\le t\le\pi\}$ forms a leaf-shaped domain, which serves as the defining target region for the class $\mathcal C_{\arcsin}$.}
	\label{fig:leaf-domain}
\end{figure}

\begin{table}[H]
	\centering
	\caption{Ma--Minda convex classes $\mathcal C(\varphi)$ associated with special mapping functions $\varphi$.}
	\label{tab:related-classes}
	\begin{tabular}{@{}lll@{}}
		\toprule
		Author(s) & $\varphi(z)$ & Class \\
		\midrule
		Janowski (1970) & $\dfrac{1+Az}{1+Bz}$ & $\mathcal C[A,B]$ \\
		Mendiratta et al. (2015) & $e^{z}$ & $\mathcal C_{e}$ \\
		Alimohammadi et al. (2021) & $1+\alpha z,\ 0<\alpha\le 1$ & $\mathscr C(1+\alpha z)$ \\
		Arora--Kumar (2022) & $1+\operatorname{arcsinh}z$ & $S^{*}_{\rho}$ (starlike) \\
		\textbf{Present paper} & $1+\dfrac{2}{\pi}\arcsin z$ & $\mathcal C_{\arcsin}$ \\
		\bottomrule
	\end{tabular}
\end{table}
\noindent{\bf\underline{Logarithmic and inverse logarithmic coefficients.}}
Let $f\in\mathcal{S}$. Two important coefficient sequences associated with $f$ are the logarithmic coefficients $\{\gamma_n\}$ and the inverse logarithmic coefficients $\{\Gamma_n\}$. The logarithmic coefficients are defined by
\bea\label{e1.2}
F_f(z)=\log\!\left(\frac{f(z)}{z}\right)
=2\sum_{n=1}^{\infty}\gamma_n(f)z^n,
\qquad z\in\Omega\setminus\{0\},\quad \log 1=0.\eea
Now from \eqref{e1.2} the inverse logarithmic coefficients are defined analogously through the inverse function $f^{-1}$ and coefficients $\{\Gamma_n\}$. The definitions together with the first few coefficient identities are summarized in Figure~\ref{fig:logcoeff}.

\begin{figure}[H]
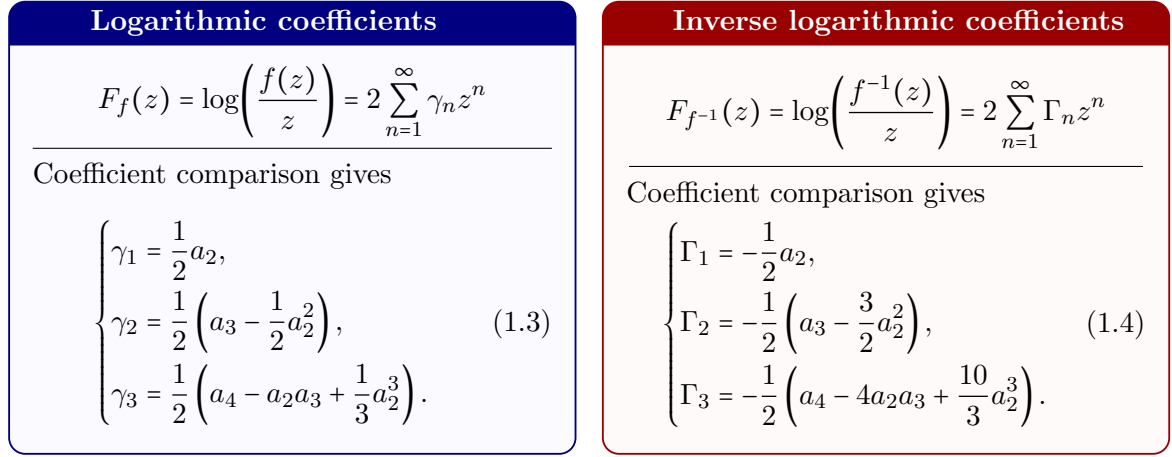

	\centering
	
	\begin{minipage}[t]{0.485\textwidth}
		\begin{tcolorbox}[
			height=6.0cm,
			colback=blue!2,
			colframe=blue!50!black,
			title=\bfseries\;\;\;\;\; Logarithmic coefficients,
			boxrule=0.7pt,
			arc=2mm,
			left=2mm,right=2mm,top=1mm,bottom=1mm]
			
			\[ F_f(z)=
			\log\!\left(\frac{f(z)}{z}\right)
			=
			2\sum_{n=1}^{\infty}\gamma_nz^n
			\]
			\noindent\makebox[\linewidth]{\rule{\linewidth}{0.4pt}}
			Coefficient comparison gives			
			
			\be
				\left\{
				\begin{aligned} 
					\label{e1.3}&\gamma_1=\frac{1}{2}a_2,\\ &\gamma_2=\frac{1}{2}\left(a_3-\frac{1}{2}a_2^2\right),\\ &\gamma_3=\frac{1}{2}\left(a_4-a_2a_3+\frac{1}{3}a_2^3\right).
					%,\\
				\end{aligned}
				\right.
			\ee

		\end{tcolorbox}
	\end{minipage}
	\hfill
	\begin{minipage}[t]{0.485\textwidth}
		\begin{tcolorbox}[
			height=6.0cm,
			colback=red!2,
			colframe=red!60!black,
			title=\bfseries \;\;\;\;\;Inverse logarithmic coefficients,
			boxrule=0.7pt,
			arc=2mm,
			left=2mm,right=2mm,top=1mm,bottom=1mm]
			
			\[F_{f^{-1}}(z)=
			\log\!\left(\frac{f^{-1}(z)}{z}\right)
			=
			2\sum_{n=1}^{\infty}\Gamma_nz^n
			\]
		\noindent\makebox[\linewidth]{\rule{0.99\linewidth}{0.4pt}}
			\;\;\;\;\ Coefficient comparison gives
			\begin{equation}
				\left\{
				\begin{aligned} 
					\label{e1.4}&\Gamma_1=-\frac{1}{2}a_2,\\ &\Gamma_2=-\frac{1}{2}\left(a_3-\frac{3}{2}a_2^2\right),\\ &\Gamma_3=-\frac{1}{2}\left(a_4-4a_2a_3+\frac{10}{3}a_2^3\right).
					%\\
					%&\Gamma_4=-\frac{1}{2}\left(a_5-5a_2a_4+15a_2^2a_3-\frac{5}{2}a_3^2-\frac{35}{4}a_2^4\right).
				\end{aligned}
				\right.
			\end{equation}
			
		\end{tcolorbox}
	\end{minipage}
	\vspace{.5cc}
	\caption{Definitions and initial coefficients of the logarithmic and inverse logarithmic coefficients.}
	
	\label{fig:logcoeff}
\end{figure}
The inverse logarithmic coefficients were introduced by
Ponnusamy \emph{et al.}~\cite{PSW_RM}, who proved the sharp estimate $\Gamma_n|\le\frac1{2n}\binom{2n}{n},$
with equality only for the Koebe function and its rotations.

The study of logarithmic coefficients has become an active area of geometric function theory. Sharp estimates have been established for various subclasses of $\mathcal S$
(see \cite{Ali_Vasudevarao_PAMS_2018, 	Ali_Allu_BAMS,
	CKKLS_	RACSAM_2020,	Zaprawa_BSMM_2021,	Thomas_PAMS_2016,
	Krgar_JA_2019,	Ponnusamy_BSM_2021,	Ponnusamy_JAMS_2020}).
\subsection*{\underline{Hankel determinants}}

For a function $f\in\mathcal A$ given by \eqref{e-1.1},
Pommerenke \cite{Pom_1967} introduced the Hankel determinant,
which provides higher-order coefficient invariants for univalent
functions. Besides the classical coefficients $\{a_n\}$,
analogous Hankel determinants can also be defined for the logarithmic
coefficients $\{\gamma_n\}$ of $f$ and
$\{\Gamma_n\}$ of its inverse $f^{-1}$.
\begin{align}
	H_q(n)
	&=
	\begin{vmatrix}
		a_n&a_{n+1}&\cdots&a_{n+q-1}\\
		a_{n+1}&a_{n+2}&\cdots&a_{n+q}\\
		\vdots&\vdots&\ddots&\vdots\\
		a_{n+q-1}&a_{n+q}&\cdots&a_{n+2q-2}
	\end{vmatrix},
	\label{e1.5}
	\\[3mm]
	H_{q,n}(F_f/2)
	&=
	\begin{vmatrix}
		\gamma_n&\gamma_{n+1}&\cdots&\gamma_{n+q-1}\\
		\gamma_{n+1}&\gamma_{n+2}&\cdots&\gamma_{n+q}\\
		\vdots&\vdots&\ddots&\vdots\\
		\gamma_{n+q-1}&\gamma_{n+q}&\cdots&\gamma_{n+2q-2}
	\end{vmatrix},
	\label{e1.6}
	\\[3mm]
	H_{q,n}(F_{f^{-1}}/2)
	&=
	\begin{vmatrix}
		\Gamma_n&\Gamma_{n+1}&\cdots&\Gamma_{n+q-1}\\
		\Gamma_{n+1}&\Gamma_{n+2}&\cdots&\Gamma_{n+q}\\
		\vdots&\vdots&\ddots&\vdots\\
		\Gamma_{n+q-1}&\Gamma_{n+q}&\cdots&\Gamma_{n+2q-2}
	\end{vmatrix}.
	\label{e1.7}
\end{align}

The lowest-order cases used throughout this paper are
\[
H_{2,2}(f)=a_2a_4-a_3^2,\qquad
H_{2,1}(F_f/2)=\gamma_1\gamma_3-\gamma_2^2,\qquad
H_{2,1}(F_{f^{-1}}/2)=\Gamma_1\Gamma_3-\Gamma_2^2.
\]

The investigation of Hankel determinants involving logarithmic
coefficients was initiated by Kowalczyk--Lecko
\cite{KL_BAMS_2022, KL_RACSAM_2023}.
Their work opened a fruitful direction in coefficient theory,
leading to numerous subsequent developments
(see
\cite{Allu_Shaji_bams_2025, ASS,
	Sabir_Ali_2025,
	MK_BIMS_2023,
	Eker et al,
	Lecko et al,
	Majumder et al._JA_2026}).
%%%%%%%%%%%%%%%%%%%%%%%%%%%%%%%%%%%%%%%%%%%%%%%%%%%%%%%%%%%%%%%%%%%%
\\
\noindent\textbf{The main contributions of this paper are as follows:}
\begin{itemize}
	\item Sharp bounds for the initial Taylor coefficients $a_2, a_3, a_4$ for functions in $\mathcal{C}_{\arcsin}$ (Theorem \ref{t3.1});
	\item Sharp estimates for logarithmic coefficients $\gamma_1,\gamma_2,\gamma_3,\gamma_4$ (Theorem \ref{t3.2});
	\item Sharp estimates for inverse logarithmic coefficients $\Gamma_1,\Gamma_2,\Gamma_3$ (Theorem \ref{t3.3});
	\item Sharp bounds for the difference $|\gamma_2|-|\gamma_1|$ (Theorem \ref{t3.4});
	\item Sharp bounds for the difference $|\Gamma_2|-|\Gamma_1|$ (Theorem \ref{t3.5}).
	\item Sharp estimate for the second Hankel determinant $H_{2,2}(f)$ (Theorem \ref{t3.6});
	\item Sharp bounds for the Hankel determinant $H_{2,1}(F_f/2)$ involving logarithmic coefficients (Theorem \ref{t3.7});
	\item Sharp bounds for the Hankel determinant $H_{2,1}(F_{f^{-1}}/2)$ involving inverse logarithmic coefficients (Theorem \ref{t3.8});
	\item Sharp bounds for the generalized Zalcman functional $\mid a_2a_3-a_4\mid$ (Theorem \ref{t3.9}).
	\item Sharp bounds of generalized Fekete--Szeg\"{o} functional (Theorem \ref{t3.10})
	
\end{itemize}
All estimates are sharp, and extremal functions are explicitly identified in each case.

%The paper is organized as follows: In Section 2 we introduced some necessary lemmas. Section 3 contains main results and its divided subsections. In Subsection 3.1 we established sharp bounds of initial Taylor's coefficients and logarithmic coefficients of functions and inverse logarithmic coefficients. In Subsection 3.2 we establish sharp bounds of coefficient differences. Subsection 3.3 studied the sharp bounds of the Hankel determinants. In subsection 3.4 we find sharp bounds of the initial generalized Zalcman conjecture, and in subsection 3.5 we find the sharp bound of the Fekete-Szeg\"{o} functional. 
\section{\bf{Lemmas}}
		Let $\mathcal{P}$ be the class of all analytic functions in the unit disk $\Omega$ such that $p(0)=1$ and $\Re(p(z))>0$ for all $z\in \Omega$. Every $p\in \mathcal{P}$ then has the series representation 
		\bea\label{e2.1} p(z)=1+\sum_{n=1}^{\infty} p_nz^n, \ \ z\in\Omega. \eea
		Function in $\mathcal{P}$ are referred to as Carath\'{e}odory functions. It will known that for $p\in\mathcal{P}$, the coefficients $\left(\text{see \cite{PL_D_1983}}\right)$ satisfy the sharp bound $\mid p_n\mid\leq 2$ for all $n\ge 1$. The Carath\'{e}odory class $\mathcal{P}$ and its coefficient bounds play a fundamental  role in deriving sharp estimates in geometric function theory.
  \begin{lem}\label{l2.1}\cite{Ma+Minda_1992}
		If $h\in\mathcal{P}$ be given by \emph{(\ref{e2.1})}, then
		\[
		|p_2-{v} p_1^2| \le 
		\begin{cases}
			-4{v}+2 & \text{if } {v}<0, \\
			2 & \text{if } 0\le{v}\le 1,\\
			4{v}-2 & \text{if } {v}>1.
		\end{cases}
		\]
		Moreover, for ${v}<0$ or ${v}>1$ equality holds if and only if $h(z)=\frac{1+z}{1-z}$ one of its rotations. \\Also, for $0<{v}<1$ equality holds if and only if $h(z)=\frac{1+z^2}{1-z^2}$ or one of its rotations. 
	\end{lem}
	\begin{lem}\label{l2.2}\cite{Ali_BMSS_2003}
		If $h\in\mathcal{P}$ be given by \emph{(\ref{e2.1})} with $0\le B\le 1$ and $2B(2B-1)\le D\le B$. Then 
		\beas |p_3-2Bp_1p_2+Dp_1^3|\le2.\eeas.
	\end{lem}
	%\begin{lem}\label{l2.3}\cite{Ravichandran_Verma_CRMAA_2015}
	%	If $h\in\mathcal{P}$ be given by \emph{(\ref{e2.1})}. If $\beta$, $\gamma$, $\delta$ and $\xi$ satisfy $0<\beta<1$, $0<\xi<1$, and 
	%	\beas 8\xi(1-\xi)\big\{(\beta\gamma-2\delta)^2+(\beta(\xi+\beta)-\gamma)^2\big\}+\beta(1-\beta)(\gamma-2\beta\xi)^2\le 4\beta^2(1-\beta)^2\xi(1-
	%	\xi),\eeas
	%	then \beas\mid\delta p_1^4+\xi p_2^2+2\beta p_1p_3-\frac{3}{2}\gamma p_1^2p_2-p_4 \mid\le 2. \eeas
	%\end{lem}
	\begin{lem}\label{l2.4} \cite[Lemma 2.4]{CKL1} If $h\in\mathcal{P}$ is of the form \emph{(\ref{e2.1})}, then
		\bea
		\label{e2.2}p_1 =2u_1,\eea
		\bea\label{e2.3} p_2=2u_1^2 + 2(1 - u_1^2)u_2\eea
		and
		\bea
		\label{e2.4} p_3 = 2u_1^3+4(1-u_1^2)u_1u_2 - 2(1 - u_1^2)u_1u_2^2 + 2(1 - u_1^2)(1 - |u_2|^2)u_3
		\eea
		for some $u_1, u_2, u_3 \in{\ol \Omega}:= \{z \in \mathbb{C}: |z| \leq 1 \}$.
		
		\medskip
		For $u_1 \in \mathbb{T}:= \{z \in \mathbb{C}: |z| = 1 \}$, there is a unique function $p\in \mathcal{P}$ with $p_1$  as in \emph{(\ref{e2.2})}, namely
		\[p(z)=\frac{1+u_1 z}{1 - u_1 z}, \quad z \in \Omega.\]
		
		\medskip
		For $u_1 \in \Omega$ and $\tau_2 \in \mathbb{T}$, there is a unique function $h\in \mathcal{P}$ with $p_1$ and $p_2$ as in \emph{(\ref{e2.2})} and \emph{(\ref{e2.3})}, namely
		\[p(z) = \frac{1+(\ol u_1 u_2+u_1)z + u_2 z^2}{1+(\ol u_1 u_2 - u_1)z - u_2 z^2}, \quad z \in \Omega.\]

		\medskip
		For $u_1, u_2 \in \Omega$ and $u_3 \in \mathbb{T}$, there is a unique function $p\in\mathcal{P}$ with $p_1$, $p_2$ and $p_3$  as in \emph{(\ref{e2.2})--(\ref{e2.4})}, namely
		\[p(z)=\frac{1 + (\ol u_2 u_3 + \ol u_1 u_2 + u_1)z+(\ol u_1 u_3+u_1\ol u_2\tau_3+u_2)z^2+u_3z^3}{1+(\ol u_2 u_3+\ol u_1u_2-u_1)z+(\ol u_1u_3-u_1\ol u_2u_3-u_2)z^2-u_3z^3},\;\;z\in\Omega.\]
	\end{lem}

	\medskip
	Following well-known result is due to Choi et al. \cite{CKS1}.
	\begin{lem}\label{l2.5}\cite{CKS1} Let $A$, $B$, $C$  be real numbers and let
		\[\Psi(A, B, C):= \max\limits_{z\in \ol{{\Omega}}}\left\lbrace |A+Bz+Cz^2|+1-|z|^2\right\rbrace.\]
		
		\begin{enumerate} 
			\item[\emph{(i)}] If $AC\geq 0$, then
			\[\Psi(A, B, C) =
			\begin{cases}
				|A|+|B|+|C|, & \text{if}\;\;\; |B|\geq 2(1-|C|), \\
				1+|A|+\frac{B^2}{4(1-|C|)}, &\text{if}\;\;\; |B|<2(1-|C|).
			\end{cases}
			\]
			\item[\emph{(ii)}] If $AC<0$, then 
			\[\Psi(A,B,C)=
			\begin{cases}
				1-|A|+\frac{B^2}{4(1-|C|)}, &\text{if}\;\;\; -4AC(C^{-2}-1) \leq B^2\; \text{and}\; |B|<2(1-|C|), \\
				1+|A|+\frac{B^2}{4(1+|C|)}, &\text{if}\;\;\; B^2<\min\left\{4(1+|C|)^2, -4AC(C^{-2}-1) \right\}, \\
				R(A,B,C), &\text{otherwise},
			\end{cases}
			\]
			where
			\[R(A,B,C):=
			\begin{cases}
				|A|+|B|-|C|, & \text{if}\;\;\; |C|(|B|+4|A|) \leq |AB|, \\
				-|A|+|B|+|C|, & \text{if}\;\;\; |AB|\leq |C|(|B|-4|A|), \\
				(|C|+|A| )\sqrt{1-\frac{B^2}{4AC}}, &\text{otherwise}.
			\end{cases}
			\]
		\end{enumerate} 
	\end{lem}
	\begin{lem}\label{l2.6}\cite{Sim_Thomas_Symmetry_2020}
		Let $J$, $K$, and $L$ be numbers such that $J\geq 0$, $K\in\mathbb{C}$, and
		$L\in\mathbb{R}$. Let $h\in\mathcal{P}$ be of the form \eqref{e2.1} and define
		a function by
		\[
		\Phi(p_1,p_2)
		=
		\left|Kp_1^{2}+Lp_2\right|
		-
		J|p_1|.
		\]
		
		Then
		\[
		\Phi(p_1,p_2)
		\le
		\begin{cases}
			|4K+2L|-2J,
			& \text{if } |2K+L|\ge |L|+J,\\[2mm]
			2|L|,
			& \text{otherwise}.
		\end{cases}
		\]
		
		and
		
		\[
		-\Phi(p_1,p_2)
		\le
		\begin{cases}
			2J-M,
			& \text{when } J\ge M+2|L|,\\[3mm]
			2J\sqrt{\dfrac{2|L|}{M+2|L|}},
			& \text{when } J^{2}\le 2|L|(M+2|L|),\\[4mm]
			2|L|+\dfrac{J^{2}}{M+2|L|},
			& \text{otherwise},
		\end{cases}
		\]
		
		where
		\[
		M=|4K+2L|.
		\]
	\end{lem}
    \section{\bf Main results}
    \subsection{\bf Coefficient bounds}
    \begin{theo}\label{t3.1}
			Let $f(z)=z+a_2z^2+a_3z^3+\cdots\in\mathcal{C}_{\arcsin} $. Then
		$$
		|a_2|\le \frac{1}{\pi}, \ |a_3| \le \frac{1}{3\pi}, \ |a_4| \le \frac{1}{6\pi}.
		$$
		All inequalities are sharp.
	\end{theo}
    \begin{proof}
    	Let $f\in\mathcal{C}_{\arcsin}$, then by the definition of the class there exists Schwarz function $\omega(z)$ such that
        \bea\label{e3.1} \frac{zf^{\prime\prime}(z)}{f^{\prime}(z)}= \frac{2}{\pi}\operatorname{arcsin}(\omega (z)). \eea 
			Suppose that $\omega(z)=\frac{p(z)-1}{p(z)+1}$, where $p\in\mathcal{P}$ is defined in (\ref{e2.1}), then we have 
		\bea\label{e3.2} 
		\omega(z)=\frac{p_1}{2}z+\left(\frac{p_2}{2}-\frac{p_1^2}{4}\right)z^2+\left(\frac{p_3}{2}-\frac{p_1p_2}{2}+\frac{p_1^3}{8}\right)z^3+\left(\frac{p_4}{2}-\frac{p_1p_3}{2}-\frac{p_2^2}{4}+\frac{3p_1^2p_2}{8}-\frac{p_1^4}{16}\right)z^4+\cdots.
		\eea
		Substituting this expression $\omega(z)$ along with (\ref{e-1.1}) and (\ref{e3.2}) into (\ref{e3.1}) we have 
		\bea 
		\label{e3.3}a_2&=&\frac{1}{2\pi}p_1,\\ 
		\label{e3.4}a_3&=&\frac{1}{6\pi}\left(p_2-\left(\frac{\pi-2}{2\pi}\right)p_1^2\right),\\
		\label{e3.5}a_4&=&\frac{1}{12\pi}\left(p_3-\left(\frac{2\pi-3}{2\pi}\right)p_1p_2+\left(\frac{7\pi^2-18\pi+12}{24\pi^2}\right)p_1^3\right),
		%\\
	%	\label{e5.6}a_5&=&\frac{3}{16}c_4+\frac{1}{32}c_2^2+\frac{7}{64}c_1c_3+\frac{1}{128}c_1^2c_2-\frac{1}{6144}c_1^4.
		\eea
        {\bf A.} Using \eqref{e3.3} we obtain $|a_2|=\frac{1}{2\pi}|p_1|\le\frac{1}{\pi}$.\\
        To show the sharpness of the inequality, we consider the function $p(z)=\frac{1+z}{1-z}$, hence $\omega(z)=z$. Let $f_1\in\mathcal C_{\arcsin}$
        be the function determined by \eqref{e3.1}
        and its solution
       \bea\label{f_1} f_1(z)=\int_{0}^{z}\exp\left(\frac{2}{\pi}\int_{0}^{s}\frac{\operatorname{arcsin}t}{t}\,dt\right)\,ds=z+\frac{1}{\pi}z^2+\frac{2}{3\pi^2}z^3+\left( \frac{1}{36\pi}+\frac{1}{3\pi^3}\right)z^4+\cdots.\eea
       Thus $a_1=1$, $a_2=\frac{1}{\pi}$, then we have $| a_2|=\frac{1}{\pi}$.
        \\
        {\bf B.} Applying Lemma~\ref{l2.1} to \eqref{e3.4},  we get
        $|a_3| \le \frac{1}{6\pi}\left| p_2-\left(\frac{\pi-2}{2\pi}\right)p_1^2\right|\le\frac{1}{3\pi}$.\\
        For sharpness, we choose $p(z)=\frac{1+z^2}{1-z^2}$. It follows that $\omega(z)=z^2$, Let $f_2(z)\in\mathcal{C}_{\arcsin}$, satisfying differential equation (\ref{e3.1}) and its solution
       \bea\label{f_2} f_2(z)=\int_{0}^{z}\exp\left(\frac{2}{\pi}\int_{0}^{s}\frac{\operatorname{arcsin}t^2}{t}\,dt\right)\,ds=z+\frac{1}{3\pi}z^3+\frac{1}{10\pi^2}z^5+\cdots.\eea
       Therefore, we have $|a_3|=\frac{1}{3\pi}$.\\
        {\bf C.} From (\ref{e3.5}) we have 
        \bea\label{e3.6} |a_4| &=&\frac{1}{12\pi}\left| p_3-\left(\frac{2\pi-3}{2\pi}\right)p_1p_2+\left(\frac{7\pi^2-18\pi+12}{24\pi^2}\right)p_1^3\right|\nonumber\\&=&\frac{1}{12\pi}\left|p_3-2Bp_1p_2+Dp_1^3 \right|,\eea
        where $B=\frac{2\pi-3}{4\pi}\approx 0.26127$ and $D=\frac{7\pi^2-18\pi+12}{24\pi^2}\approx 0.10359$.\par 
   As illustrated in Figure~\ref{Fig:BDRegion},
   the point $(B,D)$ lies in the admissible region satisfying
   $2B(2B-1)\le D\le B$. Therefore, Lemma~\ref{l2.2} applies and it follows that $2B(2B-1)\le D\le B$, thus by Lemma \ref{l2.2} from (\ref{e3.6}) we obtain \beas |a_4| \le \frac{1}{6\pi}.\eeas
        
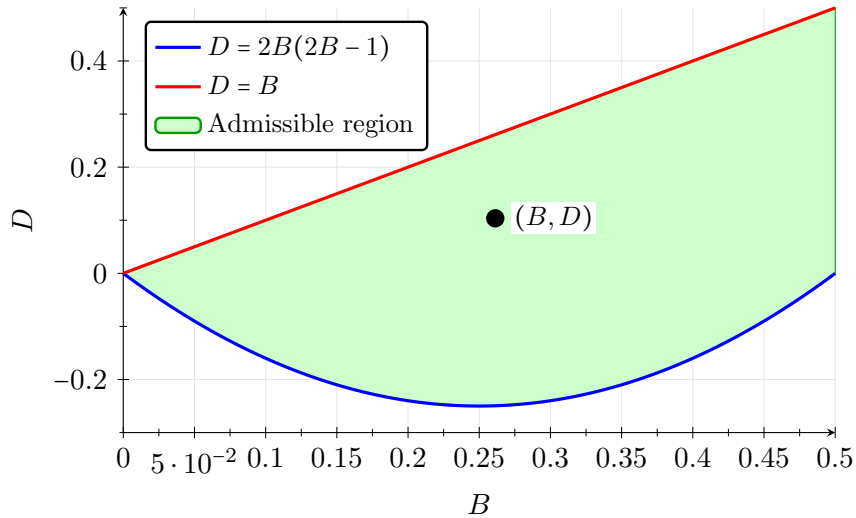
\begin{figure}[H]
	\centering
	\begin{tikzpicture}
		\begin{axis}[
			width=11cm,
			height=7.2cm,
			xmin=0,
			xmax=0.5,
			ymin=-0.30,
			ymax=0.50,
			xlabel={$B$},
			ylabel={$D$},
			axis lines=left,
			grid=major,
			major grid style={gray!20},
			minor tick num=1,
			tick style={black},
			line width=0.9pt,
			clip=false,
			legend style={
				at={(0.03,0.97)},
				anchor=north west,
				draw=black,
				fill=white,
				rounded corners=2pt,
				font=\small,
				cells={anchor=west},
				row sep=2pt,
				inner sep=3pt
			}
			]
			
			\addplot[domain=0:0.5, samples=300, very thick, blue]
			({x},{2*x*(2*x-1)});
			\addlegendentry{$D=2B(2B-1)$}

			\addplot[domain=0:0.5, samples=2, very thick, red]
			({x},{x});
			\addlegendentry{$D=B$}

			\addplot[domain=0:0.5, samples=300, draw=none, name path=Upper, forget plot]
			({x},{x});
			
			\addplot[domain=0:0.5, samples=300, draw=none, name path=Lower, forget plot]
			({x},{2*x*(2*x-1)});
			
			\addplot[draw=green!60!black, fill=green!20, forget plot]
			fill between[of=Upper and Lower];

			\addlegendimage{area legend, draw=green!60!black, fill=green!20}
			\addlegendentry{Admissible region}

			\addplot[only marks, mark=*, mark size=3pt, black]
			coordinates{(0.2612675854,0.1035948438)};
			
			\node[fill=white, inner sep=1pt, anchor=west, font=\small]
			at (axis cs:0.2722675854,0.1035948438) {$\left(B,D\right)$};
			
		\end{axis}
	\end{tikzpicture}
	
	\caption{The shaded region represents the admissible parameter domain
		$\{(B,D):2B(2B-1)\le D\le B\}$ required in Lemma~\ref{l2.2}. The point
		$\left(B,D\right)=\left(\frac{2\pi-3}{4\pi},
		\frac{7\pi^2-18\pi+12}{24\pi^2}\right)$ lies inside this region,
		confirming that Lemma~\ref{l2.2} is applicable.}
	
	\label{Fig:BDRegion}
\end{figure}       
        For sharpness of the inequality, we consider $p(z)=\frac{1+z^3}{1-z^3}$. It follows that $\omega(z)=z^3$. Let $f_3(z)\in\mathcal{C}_{\arcsin}$, satisfying differential equation (\ref{e3.1}) and its solution
       \bea\label{f_3} f_3(z)=\int_{0}^{z}\exp\left(\frac{2}{\pi}\int_{0}^{s}\frac{\operatorname{arcsin}t^3}{t}\,dt\right)\,ds=z+\frac{1}{6\pi}z^4+\frac{2}{63\pi^2}z^7+\cdots.\eea
       Hence, we have $|a_4|=\frac{1}{6\pi}$.
    \end{proof}
    \begin{table}[H]
    	\centering
    	\renewcommand{\arraystretch}{1.25}
    	\caption{Sharp coefficient bounds together with the corresponding extremal functions.}
    	\label{Tab:SharpCoeff}
    	\begin{tabular}{|c|c|c|c|}
    		\hline
    		Coefficient & Sharp bound & Extremal Carath\'eodory function & Corresponding $\omega(z)$\\
    		\hline
    		$|a_2|$ &
    		$\displaystyle \frac1\pi$
    		&
    		$\displaystyle \frac{1+z}{1-z}$
    		&
    		$z$
    		\\
    		\hline
    		$|a_3|$
    		&
    		$\displaystyle \frac1{3\pi}$
    		&
    		$\displaystyle \frac{1+z^2}{1-z^2}$
    		&
    		$z^2$
    		\\
    		\hline
    		$|a_4|$
    		&
    		$\displaystyle \frac1{6\pi}$
    		&
    		$\displaystyle \frac{1+z^3}{1-z^3}$
    		&
    		$z^3$
    		\\
    		\hline
    	\end{tabular}
    \end{table}
    \begin{theo}{\label{t3.2}}
	 	Let $f(z)=z+a_2z^2+a_3z^3+\cdots\in\mathcal{C}_{\arcsin} $ and $\gamma_1$, $\gamma_2$, $\gamma_3$ are given in \emph{(\ref{e1.3})}. Then
	 	$$
	 	|\gamma_1|\le \frac{1}{2\pi},\quad |\gamma_2|\le \frac{1}{6\pi}, \quad |\gamma_3|\le \frac{1}{12\pi}.
	 	$$
	 	The estimates are sharp.  \end{theo}
        \begin{proof}
            Proceeding in the same way as Theorem \ref{t3.1}, we have the expressions of $a_2,a_3,a_4$ same as (\ref{e3.3} - \ref{e3.5}) and we can get the desired estimade bounds of $|\gamma_i|$, $i=1,2,3$. Also for sharpness of the inequality we use the function $f_1$, $f_2$ and $f_3$ defined by \eqref{f_1}, \eqref{f_2} and \eqref{f_3}.
            
            %\medskip \\{\bf A.} Using (\ref{e3.3}) in (\ref{e1.3}) we have \beas |\gamma_1| =\left|\frac{1}{4\pi}p_1\right| \le \frac{1}{2\pi}. \eeas Hence, \beas|\gamma_1| \le \frac{1}{2\pi}, \eeas is  the desired estimate of bound. \medskip \\{\bf B.} Substituting the expressions of $a_2,a_3$ in $\gamma_2$ of (\ref{e1.3}) we get \beas |\gamma_2|=\frac{1}{12\pi}\left|p_2-\left(\frac{1}{2}-\frac{1}{4\pi}\right)p_1^2\right|.\eeas Using Lemma \ref{l2.1} we get \beas |\gamma_2| \le \frac{1}{6\pi}, \eeas which is the desired estimation of bound. \medskip \\{\bf C.} From (\ref{e1.3}) using (\ref{e3.3} - \ref{e3.5}) we obtain \beas |\gamma_3|&=&\frac{1}{24\pi}\left| p_3-\left(1-\frac{1}{2\pi}\right)p_1p_2+\left(\frac{7}{24}-\frac{1}{4\pi}\right)p_1^3\right|\\&=& \frac{1}{24\pi}|p_3-2Bp_1p_2+Dp_1^3|,\eeas where $B=\frac{1}{2}\left(1-\frac{1}{2\pi}\right)\approx0.4204$, $D=\left(\frac{7}{24}-\frac{1}{4\pi}\right)\approx0.2120$. Using Lemma \ref{l2.2} we have \beas |\gamma_3|\le \frac{1}{12\pi}. \eeas 
        \end{proof}
         \begin{theo}\label{t3.3}
	 	Let $f(z)=z+a_2z^2+a_3z^3+\cdots\in\mathcal{C}^{*}_{\arcsin} $ and $\Gamma_1$, $\Gamma_2$, $\Gamma_3$ are given in \emph{(\ref{e1.4})}. Then
	 	$$
	 	|\Gamma_1| \ \le \ \frac{1}{2\pi},
	 	\qquad |\Gamma_2| \ \le \ \frac{1}{6\pi},\qquad |\Gamma_3| \ \le \ \frac{1}{12\pi}.
	 	$$
	 	The estimates are sharp.  \end{theo}
        \begin{proof}
        Similar to Theorem \ref{t3.1}, the expressions for $a_2, a_3, a_4$ are identical to those in (\ref{e3.3} - \ref{e3.5}).
        \medskip\\
        {\bf{A.}} From (\ref{e1.4}) and (\ref{e3.3}), we get 
		\beas |\Gamma_1| =\frac{1}{2}|p_1| \ \le \ \frac{1}{2\pi}.\eeas
        The estimate is sharp, since equality is attained by the extremal function $f_1$ defined by \eqref{f_1}.
        \medskip\\
        {\bf{B.}} Using \eqref{e3.3}--(\ref{e3.4}) in (\ref{e1.4}) we obtain \beas |\Gamma_2|&=&\frac{1}{2}\left| a_3-\frac{3}{2}a_2^2\right|\\&=&\frac{1}{12\pi}\left| p_2-\left(\frac{2\pi+5}{4\pi}\right)p_1^2 \right|. \eeas
		Hence, using Lemma {\ref{l2.1}} we have \[|\Gamma_2| \le \frac{1}{6\pi}.\]
        For sharpness of the inequality we choose the extremal function $f_2$ defined by \eqref{f_2}.
        \medskip\\
        {\bf{C.}} Using \eqref{e3.3}--(\ref{e3.5}) in (\ref{e1.4}) we get \bea\label{e3.6.1}|\Gamma_3| &=&\left| \frac{1}{2}\left(a_4-4a_2a_3+\frac{10}{3}a_2^3\right) \right| \nonumber\\&=& \frac{1}{24\pi} \left| p_3-\left(\frac{2\pi+5}{2\pi}\right) p_1p_2+\left(\frac{7\pi^2+30\pi+36}{24\pi^2}\right)p_1^3 \right| .\eea Now, substituting the expressions of $p_i$, $i=1,2,3$ from (\ref{e2.2}-\ref{e2.4}) in (\ref{e3.6.1}) we get \bea \label{e3.6.2} |\Gamma_3|=\frac{1}{12\pi}\left|\left(\frac{\pi^2+36}{6\pi^2}\right)u_1^3-\frac{5}{\pi}(1-u_1^2) u_1u_2-(1-u_1^2)u_1u_2^2+(1-u_1^2)(1-|u_2 |^2)u_3\right|.\eea
	Here, by Lemma \ref{l2.4}, we have $u_1\in [0, 1]$ and $|u_2|\le1$, $|u_3|\le 1$. Then from (\ref{e3.6.2}) we first consider the boundary cases:
	\beas |\Gamma_3|=
	\begin{cases}
		\frac{1}{12\pi}|\left(1-|u_2|^2\right)u_3|\le \frac{1}{12\pi}\approx 0.02653, & \text{if} \ u_1=0, \\[3mm]
		\frac{1}{12\pi}\left(\frac{1}{6}+\frac{6}{\pi^2}\right)\approx 0.02055, & \ \text{if} \ u_1=1.
	\end{cases} \eeas
	 Now, for $0< u_1<1$ applying the triangle inequality to (\ref{e3.6.2}) and $|u_3|\le1$, we can deduce the following inequality
            \bea\label{e3.6.3} \nonumber|\Gamma_3 |&\le&\frac{1}{12\pi}|1-u_1^2| \left(\left| \frac{\left(\pi^2+36\right)u_1^3}{6\pi^2(1-u_1^2)}-\frac{5}{\pi}u_1u_2-u_1u_2^2\right|+(1-|u_2| ^2)\right)\\&=&\frac{1}{12\pi}|1-u_1^2|\Psi (A_1, B_1, C_1), \eea where $\Psi (A_1, B_1, C_1)=|A_1+B_1u_2+C_1u_2^2|+ (1-|u_2|^2)$ with  $ A_1=\frac{\left(\pi^2+36\right)u_1^3}{6\pi^2(1-u_1^2)}$, $B_1=-\frac{5}{\pi}u_1$, and $C_1=-u_1$. Suppose  $0<u_1<1$, then it is clear that $A_1C_1<0$. Thus, in view of Lemma~\ref{l2.5} and the values of $A_1$, $B_1$, and $C_1$ we consider that $$-4A_1C_1(C_1^{-2}-1)-4(1+|C_1|^2) = -4-\left(\frac{10\pi^2-72}{3\pi^2}\right)u_1^2 <0$$. Therefore, we have $B_1^2 < \min \{-4A_1C_1(C_1^{-2}-1), 4(1+|C_1 |^2)\}=-4A_1C_1(C_1^{-2}-1)$. Then applying Lemma \ref{l2.5}, we get \beas  \Psi(A_1,B_1,C_1)&=& 1+|A_1|+\frac{B_1^2}{4(1+|C_1 |)}\\&=&1+\frac{\left(\pi^2+36\right)u_1^3}{6\pi^2(1-u_1^2)}+\frac{25u_1^2}{4\pi^2(1+u_1)}.\eeas
    Now, from (\ref{e3.6.3}) we have $|\Gamma_3| \le \frac{1}{12\pi} \left\{\left(\frac{1}{6}-\frac{1}{4\pi^2}\right)u_1^3+\left(\frac{25}{4\pi^2}-1\right)u_1^2+1\right\}. $ Since, $0<u_1<1$, we get $|\Gamma_3| \le \frac{1}{12\pi}\left(\frac{1}{6}+\frac{6}{\pi^2}\right)$.
   Considering all the discussion we can conclude that the desired bound is $|\Gamma_3| \le \frac{1}{12\pi}$. 
     For sharpness we consider $f_3(z)=z+\frac{1}{6\pi}z^4+\frac{2}{63\pi^2}z^7+\cdots,$ defined by \eqref{f_3}. Therefore, the coefficients $a_2=0$, $a_3=0$ and $a_4=\frac{1}{6\pi}$ and we obtain the equality
     $|\Gamma_3| = \frac{1}{12\pi}. $
     \end{proof}
     \subsection{\bf Coefficient differences}
     \begin{theo}\label{t3.4}
	 	Let $f(z)=z+a_2z^2+a_3z^3+\cdots\in \mathcal{C}_{\arcsin}$ and let $\gamma_1,\gamma_2$ be given by \eqref{e1.3}. Then
	 	\beas
	 	-\frac{1}{2\pi}\left(1-\frac{1}{6\pi}\right)  \ \le \ |\gamma_2|-|\gamma_1| \ \le \ \frac{1}{6\pi}.\eeas
	 	Both inequalities are sharp.
	 \end{theo}
     \begin{proof}
         In view of \eqref{e1.3} and using
         \eqref{e3.3}--\eqref{e3.4}, we obtain  \bea \label{e3.5.1}\nonumber|\gamma_2|-|\gamma_1| &=&\bigg| \frac{a_3}{2}-\frac{a_2^2}{4}\bigg|-\bigg|\frac{a_2}{2}\bigg|\nonumber\\&=& \frac{1}{12\pi}\left|p_2-\left(\frac{2\pi-1}{4\pi}\right)p_1^2\right|-\frac{1}{4\pi}|p_1|\nonumber\\&=& \left|Kp_1^{2}+Lp_2\right|-J|p_1|=\Phi\left(p_1, p_2\right),\eea 
	 			where $K=-\frac{1}{12\pi}\left(\frac{2\pi-1}{4\pi}\right)$, $L=\frac{1}{12\pi}$ and $J=\frac{1}{4\pi}$. 
	 			To visualize the behaviour of the function
	 			$\Phi(p_1,p_2)$ appearing in
	 			\eqref{e3.5.1},	its graph is shown in
	 			Figure~\ref{fig:phi_surface}.
	 			\begin{figure}[H]
	 				\centering
	 				\includegraphics[width=0.7\textwidth]{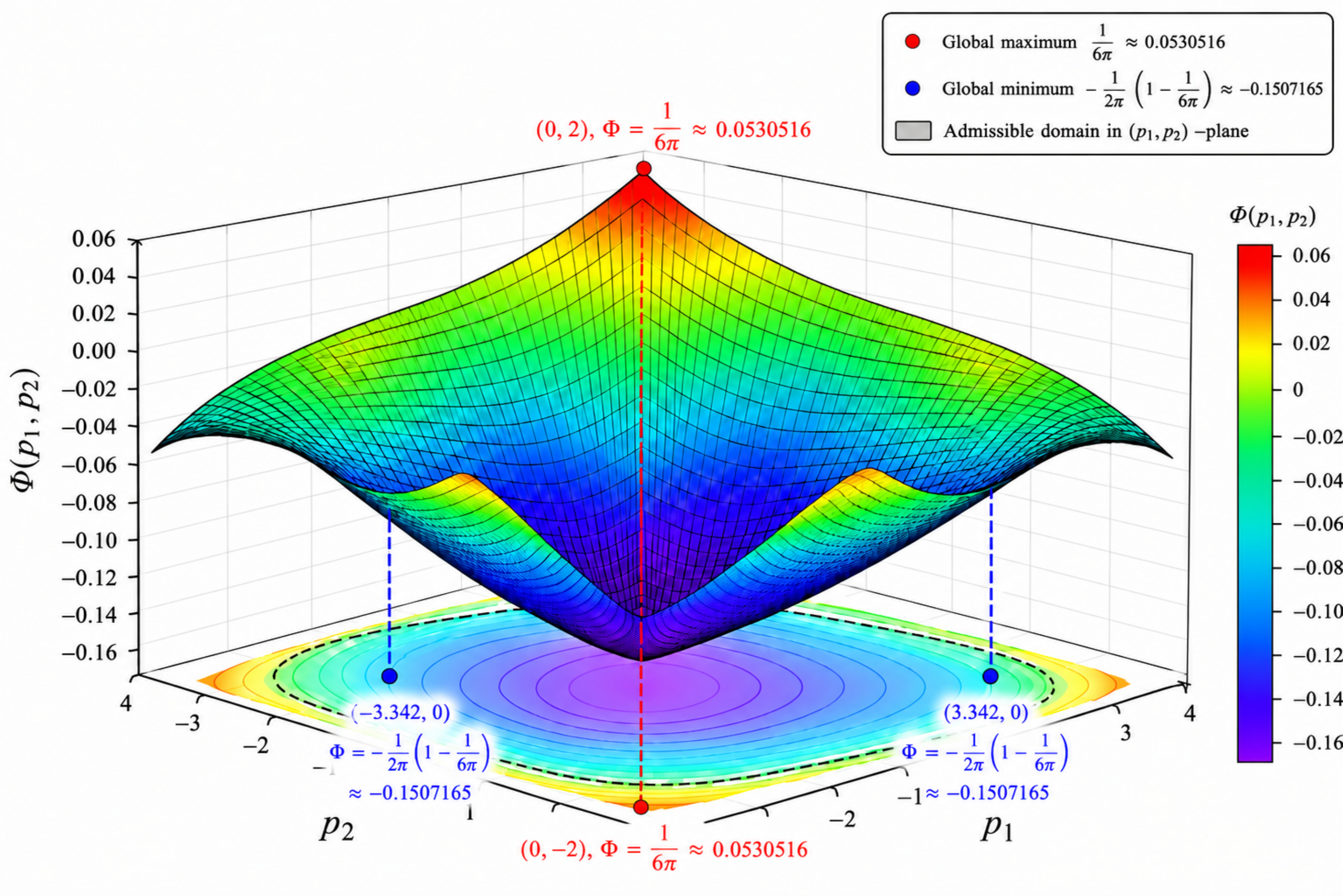}
	 					\caption{
	 					Three-dimensional coloured surface of
	 					\[
	 					\Phi(p_1,p_2)
	 					=
	 					|Kp_1^2+Lp_2|-J|p_1|,
	 					\]
	 					on the
	 					$(p_1,p_2)$-plane.
	 					The marked points indicate the extremal points at which the
	 					maximum and minimum values of
	 					$\Phi(p_1,p_2)$ are attained.
	 				}
	 				\label{fig:phi_surface}
	 				
	 			\end{figure}

	 			By a simple calculation we can see that $M=|4K+2L|=\frac{1}{12\pi^2}\approx 0.00844$, $|2K+L|=\frac{1}{24\pi^2} \approx 0.00422$ and $|L|+J=\frac{1}{3\pi} \approx 0.1061$. Since $|2K+L|<|L|+J$, then by Lemma \ref{l2.6} in (\ref{e3.5.1}) we obtain  \beas|\gamma_2|-|\gamma_1|\le 2|L| = \frac{1}{6\pi}. \eeas
                For sharpness of the upper bound we choose the extremal function $f_2$ defined by \eqref{f_2}.
                \smallskip
                                    
                Indeed, $$J-(M+2|L|)                =
                \frac1{4\pi}            -
                \left(
                \frac1{12\pi^2}+\frac1{6\pi}               \right)=               \frac{\pi-2}{12\pi^2}>0.$$
 
                Hence
                $J\ge M+2|L|$. $J\ge M+2|L|$. So by Lemma \ref{l2.6} in (\ref{e3.5.1}) we derive
                \bea\label{e3.5.2}-\Phi(p_1,p_2)=-(|\gamma_2|-|\gamma_1|) \le 2J-M=\left(\frac{1}{2\pi}-\frac{1}{12\pi^2}\right) .\eea
                Hence, the desired lower bound 
                \beas |\gamma_2|-|\gamma_1|\ge -\frac{1}{2\pi}\left(1-\frac{1}{6\pi}\right). \eeas
                 For sharpness we employ the function $f_1$ introduced by \eqref{f_1}.
     \end{proof}
     \begin{theo}\label{t3.5}
	 		Let $f(z)=z+a_2z^2+a_3z^3+\cdots\in \mathcal{C}_{\arcsin}$ and let $\Gamma_1, \ \Gamma_2$ be given by \eqref{e1.4}. Then
	 		\beas
	 		-  \left(\frac{1}{6\pi}+\frac{3}{4(5+2\pi)}\right) \le \ |\Gamma_2|-|\Gamma_1| \ \le \ \frac{1}{6\pi}.\eeas
	 		Both inequalities are sharp.
	 	\end{theo}

        \begin{proof}
            In view of (\ref{e1.4}) using (\ref{e3.3})--(\ref{e3.4}) we have \bea \label{e3.8.1}\nonumber|\Gamma_2|-|\Gamma_1| &=&\bigg|\frac{a_3}{2}-\frac{3a_2^2}{4}\bigg|-\bigg|\frac{a_2}{2}\bigg|\nonumber\\ &=& \frac{1}{12\pi}\left|p_2-\left(\frac{2\pi+5}{4\pi}\right)p_1^2\right|-\frac{1}{4\pi}|p_1|\nonumber\\&=& \left|Kp_1^{2}+Lp_2\right|-J|p_1|=\Phi\left(p_1, p_2\right),\eea 
	 			where $K=-\frac{1}{12\pi}\left(\frac{2\pi+5}{4\pi}\right)$, $L=\frac{1}{12\pi}$ and $J=\frac{1}{4\pi}$. Clearly, $M=|4K+2L|=\frac{5}{12\pi^2}\approx 0.04222$, $|2K+L|=\frac{5}{24\pi^2} \approx 0.02110$ and $|L|+J=\frac{1}{3\pi} \approx 0.1061$. Furthermore, $|2K+L|<|L|+J$, then by Lemma \ref{l2.6} in (\ref{e3.8.1}) we get  
                \beas|\Gamma_2|-|\Gamma_1|\le 2|L| = \frac{1}{6\pi}.\eeas
                To see the inequality is sharp we employ the extremal function $f_2$ introduced by \eqref{f_2}.
                \smallskip
                
                For the Lower bound it is easy to observe that $J\not\ge M+2|L|$ and $J^2\not\le2| L|\left(M+2| L|\right)$. Applying Lemma~\ref{l2.6} to
                \eqref{e3.8.1}, we obtain 
                \bea\label{e3.8.2}-\Phi(p_1,p_2)=-(|\Gamma_2|-|\Gamma_1|) \le 2|L|+\dfrac{J^{2}}{M+2|L|}=\left(\frac{1}{6\pi}+\frac{3}{4(5+2\pi)}\right) .\eea 
                For sharpness of the inequality \eqref{e3.8.2} we choose $f_4\in \mathcal{C}_{\arcsin}$ with $$p(z) = \frac{1+( \tau_1 \tau_2+\tau_1)z + \tau_2 z^2}{1+(\tau_1 \tau_2 - \tau_1)z - \tau_2 z^2},$$ where $\tau_1=\frac{3\pi}{4(5+2\pi)}$ and $\tau_2=1$. Then $p_1=2\tau_1=\frac{6\pi}{4(5+2\pi)}$ and $p_2=2\tau_1^2+2\left(1-\tau_1^2\right)\tau_2=2$. Now, from (\ref{e3.8.1}) we obtain 
                \beas |\Gamma_2|-|\Gamma_1|= \Phi(p_1, p_2)=|Kp_1^2+Lp_2|-| Jp_1|=-\left(\frac{1}{6\pi}+\frac{3}{4(5+2\pi)}\right).\eeas
                Hence, this completes the proof of this theorem.
        \end{proof}
        \subsection{\bf Hankel determinants}

    \begin{theo}\label{t3.6}
			Let $f(z)=z+a_2z^2+a_3z^3+\cdots\in\mathcal{C}_{\arcsin} $. Then the second order Hankel determinant
		$$
		|H_{2, 2}(f)|\le \frac{1}{9\pi^2}.
		$$
        The inequality is sharp.
	\end{theo}
    \begin{proof}
   Since $f\in \mathcal{C}_{\arcsin}$, it follows from (\ref{e1.5}) that the second-order Hankel determinant is given by
   \bea\label{h.3} |H_{2,2}(f)|=|a_2a_4-a_3^2|.   \eea
    Applying \eqref{e3.3}--\eqref{e3.5} to \eqref{h.3} , then we have
    \bea\label{e3.7} |H_{2,2}(f)|&=&| a_2a_4-a_3^2|\nonumber\\&=& \frac{1}{24\pi^2}\left| p_1p_3-\left(\frac{2\pi-1}{6\pi}\right)p_1^2p_2-\frac{2}{3}p_2^2+\left(\frac{3\pi^2-2\pi-4}{24\pi^2}\right)p_1^4 \right|.   \eea
   Substituting
		(\ref{e2.2})--(\ref{e2.4}) into (\ref{e3.7}), we get the expression
		\bea\label{e3.8}
		 |H_{2,2}(f)|&=&\frac{1}{6\pi^2}\bigg| \frac{(\pi^2-4)}{6\pi^2}u_1^4+\frac{1}{3\pi}(1-u_1^2)u_1^2u_2-(1-u_1^2)\frac{(2+u_1^2)}{3}u_2^2\nonumber\\&&+(1-u_1^2)(1-|u_2|^2)u_1u_3\bigg|.
		\eea
        Since $u_1\in[0, 1]$, using Lemma~\ref{l2.4} in (\ref{e3.8}) it follows that:\\
        {\bf\underline{Case-1.}} Let us choose $u_1=0$, then from (\ref{e3.8}) we obtain
        \beas |H_{2,2}(f)|=\frac{1}{9\pi^2}|u_2|^2\le \frac{1}{9\pi^2}\approx 0.01126.\eeas\\
        {\bf\underline{Case-2.}} Let us choose $u_1=1$, then from (\ref{e3.8}) we obtain
        \beas |H_{2,2}(f)|=\frac{(\pi^2-4)}{36\pi^4}\approx 0.00167.
        \eeas\\
        {\bf\underline{Case-3.}} Let us choose $0<u_1<1$. Now, applying the triangle inequality to (\ref{e3.8}), we obtain
        \bea\label{e3.9} 
		|H_{2,2}(f)|&\leq&\frac{1}{6\pi^2}\bigg(\left| \frac{(\pi^2-4)}{6\pi^2}u_1^4+\frac{1}{3\pi}(1-u_1^2)u_1^2u_2-(1-u_1^2)\frac{(2+u_1^2)}{3}u_2^2\right|\nonumber\\&&+\left|(1-u_1^2)(1-| u_2|^2)u_1u_3\right|\bigg)\nonumber\\&\le&\frac{1}{6\pi^2}u_1(1-u_1^2)\left(\left| \frac{(\pi^2-4)u_1^3}{6\pi^2(1-u_1^2)}+\frac{1}{3\pi}u_1u_2-\left(\frac{2+u_1^2}{3u_1}\right)u_2^2\right|+(1-|u_2|^2)\right)\nonumber\\&\le& \frac{1}{6\pi^2}| u_1\left(1-u_1^2\right)|\Psi\left(A_2, B_2, C_2\right),
		\eea
        where $\Psi\left(A_2, B_2, C_2\right)=| A_2+B_2u_2+C_2u_2^2|+1-|u_2|^2$ with 
        \bea\label{e3.10} A_2=\frac{(\pi^2-4)u_1^3}{6\pi^2(1-u_1^2)}, \  B_2=\frac{u_1}{3\pi} \ \text{and} \ C_2=-\frac{2+u_1^2}{3u_1}.\eea
		\smallskip
		According to the different choices of $A_2$, $B_2$, and $C_2$, we now examine the following cases under Lemma \ref{l2.5}. Consequently, $A_2C_2<0$ for all $0<u_1<1$. Thus, the following subcases arise:\\
       {\bf\underline{Subcase-3.1.}}\label{sc3.1}  Let us consider $-4A_2C_2\left(C_2^{-2}-1\right)<B_2^2$ and $|B_2| < 2(1-|C_2|)$.
	 	Now, from (\ref{e3.9}) we deduce that
	 	\beas |B_2|-2\left(1-|C_2|\right)=\frac{1}{3\pi u_1}\left((2\pi+1)u_1^2-6\pi u_1+4\pi\right)>0, \ \text{for} \ u_1\in(0, 1). \eeas
	 	Therefore, by Lemma \ref{l2.5} we get
	 	\beas \Psi(A_2, B_2, C_2)\neq1-| A_2|+\frac{B_2^2}{4\left(1-|C_2|\right)}.\eeas\\
        {\bf\underline{Subcase-3.2.}}\label{sc3.2} Let us consider $B_2^2<\min\left\{4(1+|C_2|)^2, -4A_2C_2   (C_2^{-2}-1)\right\}$. We conclude from (\ref{e3.10}), that 
         \beas
		\left(\frac{1}{C_2^2}-1\right)=-\frac{(4-u_1^2)(1-u_1^2)}{\left(2+u_1^2\right)^2}<0, \ \text{for} \ 0<u_1<1.
		\eeas
	 	Therefore, $-4A_2C_2\left({C_2^{-2}}-1\right)<0$ for all $0<u_1<1$. So, it follows that 
	 	\beas\min\left\{4(1+|C_2|)^2, -4A_2C_2(C_2^{-2}-1)\right\}=-4A_2C_2\left({C_2^{-2}}-1\right).\eeas
	 	Hence, it is clear that $B_2^2\not<\min\left\{4(1+|C_2|)^2, -4A_2C_2(C_2^{-2}-1)\right\}$.\\
	 	Therefore, by Lemma \ref{l2.5} it is clear that 
	 	\beas \Psi(A_2, B_2, C_2)\neq1+| A_2|+\frac{B_2^2}{4\left(1+|C_2|\right)}.\eeas\\
        {\bf\underline{Subcase-3.3.}}\label{sc3.3}  Let us consider $|A_2B_2|\ge|C_2| \left(|B_2|+4|A_2|\right)$. Since $0<u_1<1$, then from (\ref{e3.10}) we have
	 	\beas &&|A_2B_2|-|C_2|\left(|B_2| +4| A_2|\right)\\&&=\frac{1}{18\pi^3\left(1-u_1^2\right)}\big( (-4\pi^3+3\pi^2+16\pi-4)u_1^4-2\pi(4\pi^2-\pi-16)u_1^2-4\pi^2\big)<0.\eeas
	 	Therefore, by Lemma \ref{l2.5} we have
	 	\beas \Psi(A_2, B_2, C_2)\neq|A_2|+|B_2|-|C_2|.\eeas\\
       {\bf\underline{Subcase-3.4.}}\label{sc3.4} Let us consider $|A_2B_2|\le|C_2| \left(|B_2|-4| A_2|\right)$. So, from (\ref{e3.10}) we have
	 	\beas &&|A_2B_2|-|C_2| \left(|B_2|-4| A_2|\right)\\&&=\frac{1}{18\pi^3\left(1-u_1^2\right)}\big( (4\pi^3+3\pi^2-16\pi-4)u_1^4+2\pi(4\pi^2+\pi-16)u_1^2-4\pi^2\big).\eeas
	 	Let us choose $x=u_1^2$ and $P_1(x)=(4\pi^3+3\pi^2-16\pi-4)x^2+2\pi(4\pi^2+\pi-16)x-4\pi^2$.
The quadratic polynomial $P_1(x)$ has the unique positive root $x_0\approx0.20987$ (see Figure~\ref{Fig:P1}).
Since its leading coefficient is positive,
it follows that $$P_1(x)\le0,\qquad
0\le x\le x_0\approx0.20987.$$ 
	
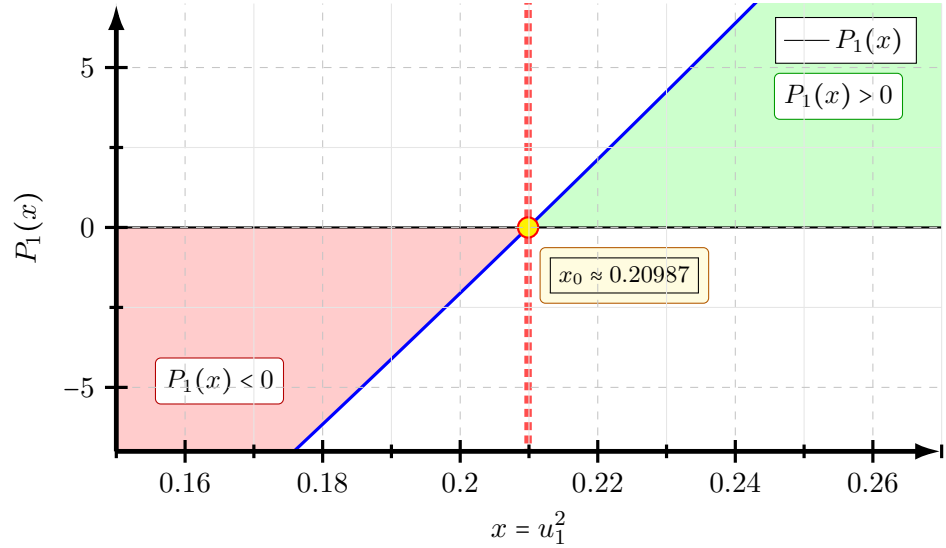
\begin{figure}[H]
	\centering
	\begin{tikzpicture}
		
		\begin{axis}[
			width=12.5cm,
			height=7.5cm,
			xmin=0.15,
			xmax=0.27,
			ymin=-7,
			ymax=7,
			xlabel={$x=u_1^2$},
			ylabel={$P_1(x)$},
			axis lines=left,
			axis on top,
			axis line style={
				black,
				line width=1.8pt,
				-{Latex[length=3.5mm,width=2.2mm]}
			},
			tick style={
				black,
				line width=1.2pt
			},
			major tick length=3mm,
			minor tick length=1.8mm,
		major grid style={
			gray!45,
			dashed
		},
		minor grid style={
			gray!20
		},
			grid=both,
			minor tick num=1,
			samples=500,
			domain=0.15:0.27,
			legend style={
				draw=black,
				fill=white,
				font=\small,
				at={(0.97,0.97)},
				anchor=north east
			},
			]

			\addplot[
			draw=none,
			name path=curveleft,
			domain=0.15:0.209868,
			samples=400
			]
			{
				(4*pi^3+3*pi^2-16*pi-4)*x^2
				+2*pi*(4*pi^2+pi-16)*x
				-4*pi^2
			};
			
			\addplot[
			draw=none,
			name path=xaxisleft,
			domain=0.15:0.209868
			]
			{0};
			
			\addplot[
			red!20
			]
			fill between[
			of=curveleft and xaxisleft
			];

			\addplot[
			draw=none,
			name path=curveright,
			domain=0.209868:0.27,
			samples=400
			]
			{
				(4*pi^3+3*pi^2-16*pi-4)*x^2
				+2*pi*(4*pi^2+pi-16)*x
				-4*pi^2
			};
			
			\addplot[
			draw=none,
			name path=xaxisright,
			domain=0.209868:0.27
			]
			{0};
			
			\addplot[
			green!20
			]
			fill between[
			of=curveright and xaxisright
			];

			\addplot[
			blue,
			very thick
			]
			{
				(4*pi^3+3*pi^2-16*pi-4)*x^2
				+2*pi*(4*pi^2+pi-16)*x
				-4*pi^2
			};
			
			\addlegendentry{$P_1(x)$}

			\addplot[
			black,
			thick
			]
			coordinates
			{
				(0.15,0)
				(0.27,0)
			};

		\addplot[
		red!70!,
		densely dashed,
		very thick,
		line width=2.8pt
		]
		coordinates
		{
			(0.209868,-8)
			(0.209868,8)
		};
			coordinates
			{
				(0.209868,-8)
				(0.209868,8)
			};
			
			\addplot[
			only marks,
			mark=*,
			mark size=3.8pt,
			fill=yellow,
			draw=red,
			line width=0.8pt
			]
			coordinates
			{
				(0.209868,0)
			};
			
			\node[
			fill=yellow!15,
			draw=orange!70!black,
			rounded corners=2pt,
			font=\footnotesize
			]
			at (axis cs:0.22387,-1.5)
			{$\boxed{x_0\approx0.20987}$};

			\node[
			font=\small,
			fill=white,
			draw=red!70!black,
			rounded corners=2pt
			]
			at (axis cs:0.165,-4.8)
			{$P_1(x)<0$};

			\node[
			font=\small,
			fill=white,
			draw=green!60!black,
			rounded corners=2pt
			]
			at (axis cs:0.255,4.1)
			{$P_1(x)>0$};
			
		\end{axis}
	\end{tikzpicture}
	
	\caption{
	Graph of $P_1(x)$. The unique positive zero is
		$x\approx0.20987$, implying
		$P_1(x)\le0$ for $0\le x\le0.20987$, as used in Subcase~3.4.
	}
	
	\label{Fig:P1}
	
\end{figure}
	
	Then, for $0<u_1=\sqrt{x_0}\le0.45812,$
	we obtain $|A_2B_2|\le|C_2| \left(|B_2|-4| A_2|\right)$.\par 
	 	Hence, using Lemma \ref{l2.5} and from (\ref{e3.9}) we obtain
	 	\bea\label{e3.11} |H_{2,2}(f)| &\le&\frac{1}{6\pi^2}u_1\left(1-u_1^2\right)\Psi(A_2, B_2, C_2)\nonumber\\&\le&\frac{1}{36\pi^4}\left(-(3\pi^2+2\pi-4)u_1^4-2\pi(\pi-1)u_1^2+4\pi^2\right)=\frac{1}{36\pi^4}P_2(u_1),  \eea
	 	where $P_2(u_1)=-(3\pi^2+2\pi-4)u_1^4-2\pi(\pi-1)u_1^2+4\pi^2$. 
	 	It is easy to deduce for $0<u_1\le 0.45821$, $P_2(u_1)$ is monotone decreasing (see Figure~\ref{Fig:P2}). 
	 	\begin{figure}[H]
	 		\centering
	 		\begin{tikzpicture}
	 			\begin{axis}[
	 				width=12cm,
	 				height=8cm,
	 				xmin=0,
	 				xmax=0.48,
	 				ymin=34,
	 				ymax=40.6,
	 				xlabel={$u_1$},
	 				ylabel={$P_2(u_1)$},
	 				axis lines=left,
	 				axis line style={very thick,-latex},
	 				tick style={very thick},
	 				grid=major,
	 				major grid style={gray!20},
	 				minor tick num=1,
	 				samples=400,
	 				domain=0:0.45821,
	 				smooth,
	 				clip=true,
	 				legend style={
	 					at={(0.75,0.70)},
	 					anchor=north west
	 				}
	 				]

	 				\addplot[
	 				ultra thick,
	 				blue!80!black
	 				]
	 				{
	 					-(3*pi^2+2*pi-4)*x^4
	 					-2*pi*(pi-1)*x^2
	 					+4*pi^2
	 				};
	 				
	 				\addlegendentry{$P_2(u_1)$}

	 				\addplot[
	 				red,
	 				densely dashed,
	 				very thick
	 				]
	 				coordinates
	 				{
	 					(0.45821,34)
	 					(0.45821,40)
	 				};
	 				
	 				\node[
	 				fill=white,
	 				font=\small,
	 				anchor=south east
	 				]
	 				at (axis cs:0.45821,37.4)
	 				{$u_1=0.45821$};

	 			\node[
	 			draw,
	 			rounded corners=2pt,
	 			fill=yellow!20,
	 			font=\small
	 			]
	 			at (axis cs:0.220,39.5)
	 			{Monotone decreasing};
	 			
	 			\draw[-latex,very thick,red]
	 			(axis cs:0.220,39.18)
	 			--
	 			(axis cs:0.355,37.75);

	 				\addplot[
	 				only marks,
	 				mark=*,
	 				mark size=2.5pt,
	 				black
	 				]
	 				coordinates
	 				{
	 					(0,39.4784)
	 				};
	 				
	 				\node[
	 				font=\small,
	 				anchor=south west
	 				]
	 				at (axis cs:0.01,39.48)
	 				{$P_2(0)=4\pi^2$};
	 				
	 			\end{axis}
	 		\end{tikzpicture}
	 		
	 		\caption{Graph of $	 			P_2(u_1)=-(3\pi^2+2\pi-4)u_1^4-2\pi(\pi-1)u_1^2+4\pi^2,$ for $0\le u_1\le0.45821$. The function is monotone decreasing, hence $P_2(u_1)\le P_2(0)=4\pi^2,$	 	which is used in \eqref{e3.11}.}
	 		\label{Fig:P2}
	 	\end{figure}
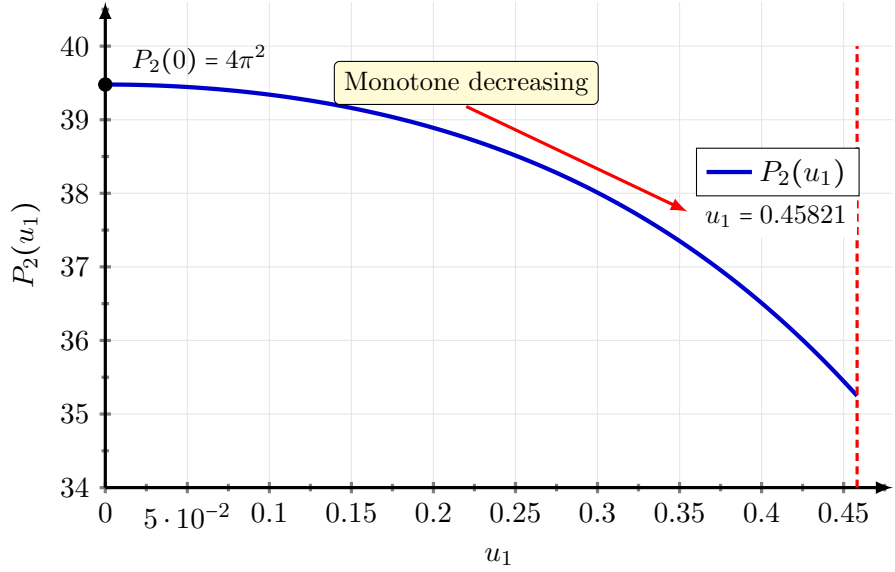
	 	Hence, from (\ref{e3.11}) we obtain
	 	\beas |H_{2,2}(f)| \le \frac{1}{36\pi^4}P_2(0)=\frac{1}{9\pi^2}\approx 0.01126.\eeas\\
	 	{\bf\underline{Subcase-3.5.}}\label{sc3.5} Let us consider $0.45821<u_1<1$. Using Lemma \ref{l2.5} and from (\ref{e3.9})--(\ref{e3.10}) we have
	 	\bea\label{e3.12} 
	 	|H_{2, 2}(f)|&\le&\frac{1}{6\pi^2}u_1(1-u_1^2)(|C_2|+|A_2| )\sqrt{1-\frac{B_2^2}{4A_2C_2}}\nonumber\\ &\le& \left(\frac{-(4+\pi^2)u_1^4-2\pi^2u_1^2+4\pi^2}{36\pi^4\sqrt{2\pi^2-8}}\right)\sqrt{\frac{4\pi^2-15+(2\pi^2-9)u_1^2}{2+u_1^2}}\nonumber\\ &\le&\frac{1}{36\pi^4(\sqrt{2\pi^2-8})}\eta_1(u_1),\eea
	 	where $\eta_1(u_1)=\left(-(4+\pi^2)u_1^4-2\pi^2u_1^2+4\pi^2\right)\sqrt{\frac{4\pi^2-15+(2\pi^2-9)u_1^2}{2+u_1^2}}$.
	 	 Differentiating with respect to $u_1$ yields
	 	\beas\eta_1^{\prime}(u_1)=\frac{u_1P_3(u_1)}{{\left(2+u_1^{2}\right)^2} \sqrt{\frac{4\pi^2-15+(2\pi^2-9)u_1^2}{2+u_1^2}}},\eeas
	 	where $$P_3(u_1)=(-8\pi^4+4\pi^2+144)u_1^6+(-40\pi^4+43\pi^2+540)u_1^4+(-64\pi^4+130\pi^2+480)u_1^2-32\pi^4+108\pi^2$$. Clearly, $P_3(u_1)<0$ for $0.45821<u_1<1$. Then we have $\eta_1^{\prime}(u_1)<0$. So $\eta_1(u_1)$ is monotone decreasing function (see Figure~\ref{fig:eta1}).
	  \begin{figure}[H]
	  	\centering
	  	\includegraphics[width=0.75\textwidth]{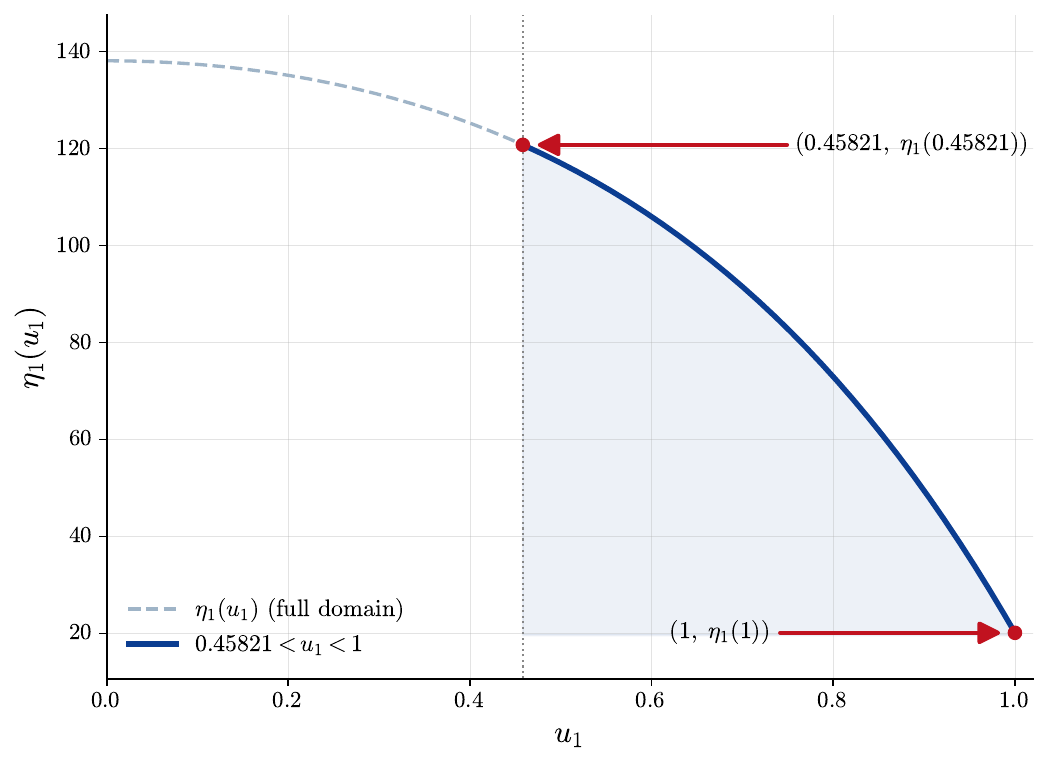}
	  	\caption{Plot of the auxiliary function $\eta_{1}(u_{1})$ on the interval $0.45821<u_{1}<1$. The graph illustrates the monotone decreasing behaviour of $\eta_{1}(u_{1})$, which is used in Subcase~3.5 to establish the estimate for $|H_{2,2}(f)|$.}
	  		\label{fig:eta1}
	  	\end{figure}
	 		 	
	 	From (\ref{e3.12}) we obtain
	 	\beas	|H_{2, 2}(f)|\le \frac{\eta_1(0.45821)}{36\pi^4\sqrt{2\pi^2-8}}\approx 0.01005. \eeas
        Comparing all the estimates obtained above, we conclude that, $$|H_{2,2}(f)|	\le
	 	\frac{1}{9\pi^2},$$
	 	which is the desired estimate.\\
	 	Equality is attained by the extremal function $f_2$ defined in \eqref{f_2}. 
	 		
    \end{proof}
\begin{theo}\label{t3.7}
	 	Let $f(z)=z+a_2z^2+a_3z^3+\cdots\in\mathcal{C}_{\arcsin} $, then
	 	\beas|H_{2, 1}(F_{f}/2)| \ \le \ \frac{1}{36\pi^2}. \eeas
	 	The inequality is sharp.
	 \end{theo}
     \begin{proof}
         Since $f\in \mathcal{C}_{\arcsin}$. Then from \eqref{e1.6}, the second-order Hankel determinant is given by 
	 	\bea\label{e3.13} H_{2, 1}(F_{f}/2)=\gamma_1\gamma_3-\gamma_2^2.\eea
	 	Using (\ref{e1.3}), (\ref{e3.3})--(\ref{e3.5}) into (\ref{e3.13}) and we get
	 	\bea\label{e3.14} H_{2,1}(F_{f}/2)&=&\frac{1}{4}\left(a_2a_4-a_3^2+\frac{1}{12}a_2^4\right)\nonumber\\ &=&\frac{1}{96\pi^2}\left( p_1p_3-\left(\frac{2\pi-1}{6\pi}\right)p_1^2p_2-\frac{2}{3}p_2^2+\left(\frac{3\pi^2-2\pi-1}{24\pi^2}\right)p_1^4 \right).\eea
	 	Substituting the Carath\'eodory coefficient representations \eqref{e2.2}--\eqref{e2.4} into \eqref{e3.14} and simplifying, we obtain 
	 	\bea\label{e3.15} 
	 	|H_{2,1}(F_{f}/2)| &=&\frac{1}{24\pi^2}\bigg| \frac{(\pi^2-1)}{6\pi^2}u_1^4+\frac{1}{3\pi}(1-u_1^2)u_1^2u_2-(1-u_1^2)\frac{(2+u_1^2)}{3}u_2^2\nonumber\\&&+(1-u_1^2)(1-|u_2|^2)u_1u_3\bigg|.
	 	\eea
	 	Since $u_1\in[0,1]$ by Lemma~\ref{l2.4}, we first consider the boundary cases:\\
        {\bf\underline{Case-1.}} Let us choose $u_1=0$, From (\ref{e3.15}), given $|u_3|<1$, we derive
        \beas |H_{2,1}(F_{f}/2)|=\frac{1}{36\pi^2}| u_2|^2\le \frac{1}{36\pi^2}\approx 0.00281.\eeas\\
        {\bf\underline{Case-2.}} Let us choose $u_1=1$, From (\ref{e3.15}), we obtain
		\beas |H_{2,1}(F_{f}/2)|=\frac{(\pi^2-1)}{144\pi^4}\approx 0.00063.
		\eeas\\
        {\bf\underline{Case-3.}} Let us choose $0<u_1<1$. The triangle inequality applied to (\ref{e3.15}) yields
        \bea\label{e3.16} 
		|H_{2,1}(F_{f}/2)|&\leq&\frac{1}{24\pi^2}u_1(1-u_1^2)\left(\left| \frac{(\pi^2-1)u_1^3}{6\pi^2(1-u_1^2)}+\frac{1}{3\pi}u_1u_2-\left(\frac{2+u_1^2}{3u_1}\right)u_2^2\right|+(1-|u_2|^2)\right)\nonumber\\&\le& \frac{1}{24\pi^2}|u_1\left(1-u_1^2\right)|\Psi\left(A_3, B_3, C_3\right),
		\eea
        where $\Psi\left(A_3, B_3, C_3\right)=|A_3+B_3u_2+C_3u_2^2|+1-|u_2|^2$ with 
        \bea\label{e3.17} A_3=\frac{(\pi^2-1)u_1^3}{6\pi^2(1-u_1^2)}, \  B_3=\frac{u_1}{3\pi} \ \text{and} \ C_3=-\frac{2+u_1^2}{3u_1}.\eea
		\smallskip
		Depending on the choices of $A_3$, $B_3$, and $C_3$, we examine the following cases under Lemma \ref{l2.5}. Clearly, $A_3C_3<0$ for all $0<u_1<1$. This leads to the following sub-cases:
        \smallskip

        Since $B_3=B_2$, $C_3=C_2$ and $A_3$ is analogous to $A_2$, the argument is identical to that in the proof of Theorem~\ref{t3.6}, particularly Subcase 3.1. to Subcase 3.5..
        \smallskip
        
        Hence, the desired bound is $|H_{2,1}(F_{f}/2)|\le\frac{1}{36\pi^2}$.\\
        To see the inequality is sharp, we choose the function $f_2$ defined by \eqref{f_2}.        
     \end{proof}
     
\begin{theo}\label{t3.8}
	Let  $f(z)=z+a_2z^2+a_3z^3+\cdots\in\mathcal{C}_{\arcsin}$, then
	\beas|H_{2, 1}(F_{f^{-1}}/2)| \ \le \ \frac{1}{36\pi^2}.\eeas
	The inequality is sharp.
\end{theo}
\begin{proof}
	Since $f\in\mathcal{C}_{\arcsin}$, we know that 
	\bea\label{e3.7.1} H_{2, 1}(F_{f^{-1}}/2)=\Gamma_1\Gamma_3-\Gamma_2^2.\eea 
	
	Using (\ref{e1.4}) and (\ref{e3.3})--(\ref{e3.5}) in (\ref{e3.7.1}), we obtain \bea\label{e3.7.2} H_{2,1}(F_{f^{-1}}/2)&=&\frac{1}{4}\left(a_2a_4-a_2^2a_3-a_3^2+\frac{13}{12}a_2^4\right)\nonumber\\ &=&\frac{1}{96\pi^2}\left(p_1p_3-\left(\frac{5+2\pi}{6\pi}\right)p_1^2p_2-\frac{2}{3}p_2^2+\left(\frac{11+10\pi+3\pi^2}{24\pi^2}\right)p_1^4\right).\eea 
	
	Using Lemma \ref{l2.4} on (\ref{e3.7.2}) we have \bea\label{e3.7.3}\nonumber H_{2,1}(F_{f^{-1}}/2) &=&\frac{1}{24\pi^2}\left\{\left(\frac{11+\pi^2}{6\pi^2}\right)u_1^4-\frac{5}{6\pi}(1-u_1^2)u_1^2u_2-\frac{1}{3}(1-u_1^2)(2+u_1^2)u_2^2\right.\\&&\left.+(1-u_1^2)(1-|u_2|^2)u_1u_3\right\}.  \eea 
	
	Here, by Lemma \ref{l2.4}, we have $u_1\in [0, 1]$ and $|u_2|\le1$, $|u_3|\le 1$.\\
	{\bf\underline{Case--1.}} Let us choose $u_1=0$. Then from (\ref{e3.7.3}) we  get \beas | H_{2,1}(F_{f^{-1}}/2)| \le \frac{1}{36\pi^2}|u_2|^2 \le \frac{1}{36\pi^2}\approx 0.00281. \eeas \\
	{\bf\underline{Case--2.}} Let us choose $u_1=1$. Then from (\ref{e3.7.3}) we  get \beas | H_{2,1}(F_{f^{-1}}/2)| = \frac{1}{144\pi^2}\left(1+\frac{11}{\pi^2}\right)\approx 0.00149.\eeas\\
	{\bf\underline{Case--3.}} Choose $0< u_1<1$. Applying the triangle inequality to (\ref{e3.7.3}), for
	$u_1\in(0,1)$ and $|u_3|\le1$, we obtain 
	\beas | H_{2,1}(F_{f^{-1}}/2)|&\le&  \frac{u_1(1-u_1^2)}{24\pi^2}\left\{\left|\frac{(11+\pi^2)u_1^3}{6\pi^2(1-u_1^2)}-\frac{5}{6\pi}u_1u_2-\frac{1}{3}\left(\frac{2+u_1^2}{u_1}\right)u_2^2\right|+(1-|u_2|^2)\right\} \\&\le&  \frac{u_1(1-u_1^2)}{24\pi^2} \Psi (A_4,B_4,C_4),\eeas
	
	where $\Psi (A_4, B_4, C_4)=|A_4+B_4u_2+C_4u_2^2|+ (1-|u_2|^2)$ and \bea\label{e3.7.4} A_4=\frac{(11+\pi^2)u_1^3}{6\pi^2(1-u_1^2)}, \;\; B_4= -\frac{5}{6\pi}u_1,\;\; C_4=-\left(\frac{2+u_1^2}{3u_1}\right).\eea In view of Lemma~\ref{l2.5}, the subsequent discussion is divided into the following cases, determined by the values of $A_4$, $B_4$, and $C_4$.
	
	\smallskip
	Suppose  $0<u_1<1$. Then it is clear that $A_4C_4<0$.
	Now applying Lemma \ref{l2.5}-(ii), we consider the following subcases:\\
	{\bf\underline{Subcase--3.1.}}  Here, by a simple calculation we can see that $|B_4| > 2(1-|C_4|)$, hence \beas \Psi(A_4,B_4,C_4)\ne 1-|A_4|+\frac{B_4^2}{4(1-|C_4|)}. \eeas\\
	{\bf\underline{Subcase--3.2.}} Let us consider the case $B_4^2<\min\left\{4(1+|C_4|)^2, -4A_4C_4(C_4^{-2}-1)\right\}$. Now, from (\ref{e3.7.4}) we can deduce that 
	\beas \left(\frac{1}{C_4^2}-1\right)=-\frac{1}{\left(u_1^2+2\right)^2}\left(u_1^2-1\right)\left(u_1^2-4\right)<0, \ \text{for} \ 0<u_1<1.\eeas
	Since we have $A_4C_4<0$, then $-4A_4C_4\left(\frac{1}{C_4^2}-1\right)<0$ for all $0<u_1<1$, it follows that 
	\beas\min\left\{4(1+|C_4|)^2, -4A_4C_4(C_4^{-2}-1)\right\}=-4A_4C_4\left({C_4^{-2}}-1\right).\eeas
	Therefore, it is clear that $B_4^2\not<\min\left\{4(1+|C_4|)^2, -4A_4C_4(C_4^{-2}-1)\right\}$.
	\par Using Lemma \ref{l2.5}-(ii) we have
	\beas \Psi(A_4, B_4, C_4)\neq1+| A_4|+\frac{B_4^2}{4\left(1+|C_4|\right)}. \eeas\\
	{\bf\underline{Subcase--3.3.}} Next consider the case $|A_4B_4|\ge|C_4| \left(|B_4|+4|A_4|\right)$. Here from (\ref{e3.7.4}) we have
	\beas &&|A_4B_4|-|C_4| \left(|B_4|+4| A_4|\right)\\&=&\frac{(-11+15\pi-8\pi^2)u_1^4-(176-10\pi+16\pi^2)u_1^2-20\pi}{36\pi^2\left(1-u_1^2\right)}<0, \ \text{for} \ 0<u_1<1.\eeas
	Hence, by Lemma \ref{l2.5} we get
	\beas \Psi(A_4, B_4, C_4)\neq|A_4|+|B_4|-|C_4|.  \eeas\\
	{\bf\underline{Subcase--3.4.}} Here, from (\ref{e3.7.4}) we can see that \beas&&|A_4B_4|-|C_4| \left(|B_4|-4| A_4|\right)\\&=& \frac{(143+15\pi+8\pi^2)u_1^4+(176+10\pi+16\pi^2)u_1^2-20\pi}{36\pi^2\left(1-u_1^2\right)}\\ &=&\frac{P_4(u_1)}{36\pi^2\left(1-u_1^2\right)},\eeas
where $P_4(u_1)=(143+15\pi+8\pi^2)u_1^4+(176+10\pi+16\pi^2)u_1^2
-20\pi.$ Setting $x=u_1^2$, we obtain the quadratic polynomial
$$P_4(x)=(143+15\pi+8\pi^2)x^2+(176+10\pi+16\pi^2)x-20\pi.$$
A straightforward computation shows that $P_4(x)$ has a unique positive zero $x_0\approx0.15443\in(0,1)$ and
$P_4(x)\le0, 0<x\le x_0.$ Since $x=u_1^2$, it follows that $u_1=\sqrt{x_0}\approx0.39298$. Therefore, $P_4(u_1)\le0, for 0<u_1\le0.39298$.

Figure~\ref{fig:alpha38} depicts the graph of $P_4(u_1)$ on $(0,1)$. It shows that $P_4(u_1)$ has a unique positive zero at $u_1\approx0.39298$ and consequently $P_4(u_1)<0, 0<u_1<0.39298.$

	\begin{figure}[H]
		\centering
		\includegraphics[width=0.65\textwidth]{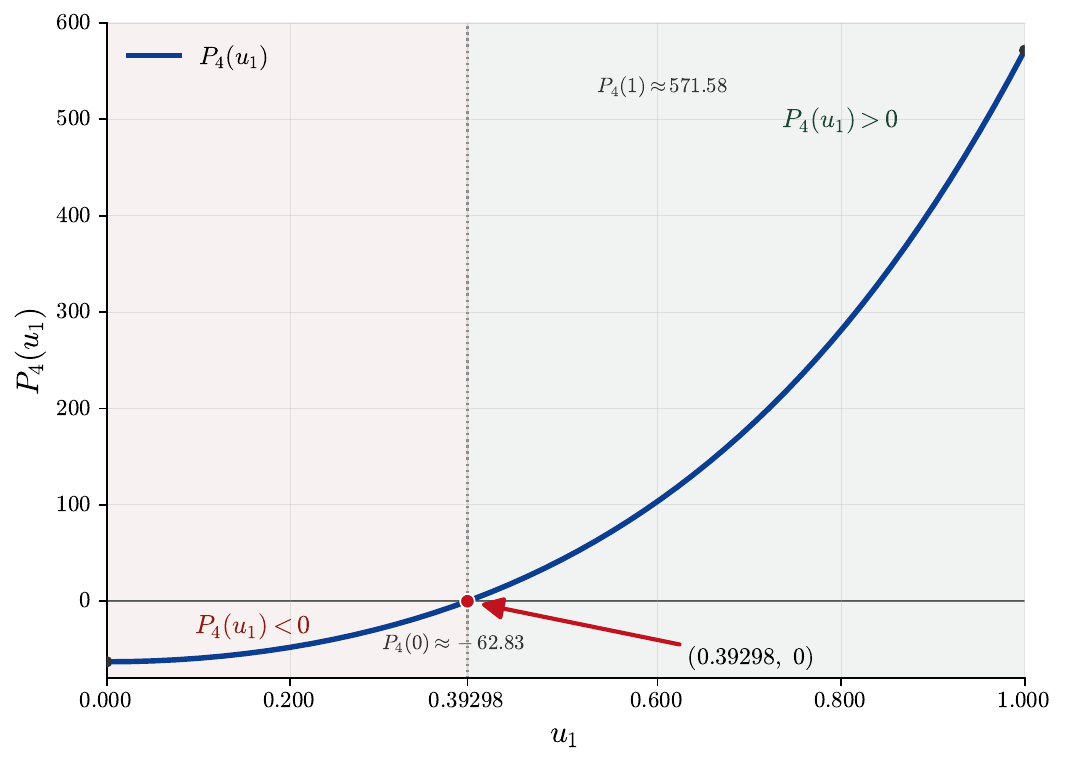}
		\caption{The polynomial $P_4(u_1)$ has a unique zero in $(0,1)$ at
			$u_1\approx0.39298$.}
		\label{fig:alpha38}
	\end{figure}	
	 
	\par Using Lemma \ref{l2.5} and from (\ref{e3.7.4}) we get
	
	\bea\label{e3.7.5} |H_{2,1}(F_{f^{-1}}/2)|&\le&\frac{u_1(1-u_1^2)}{24\pi^2} \Psi (A_4,B_4,C_4),\nonumber\\&\le&\frac{u_1(1-u_1^2)}{24\pi^2}\left(-|A_4|+|B_4|+| C_4|\right)\nonumber\\&\le&\frac{1}{144\pi^4}\left(-(3\pi^2+5\pi-11)u_1^4-(2\pi^2-5\pi)u_1^2+4\pi^2\right)\nonumber\\&=&\frac{1}{144\pi^4}P_5(u_1),  \eea
	where $P_5(u_1)=-(3\pi^2+5\pi-11)u_1^4-(2\pi^2-5\pi)u_1^2+4\pi^2$. One can verify that $P_5(u_1)$ is decreasing in $(0, 0.39298]$. Hence from (\ref{e3.7.5}) we get
	\beas |H_{2,1}(F_{f^{-1}}/2)|\le \frac{1}{36\pi^2}\approx 0.00281.\eeas\\
	{\bf\underline{Subcase--3.5.}} Let us consider $ 0.39298<u_1<1$. It remains to consider the last case in Lemma \ref{l2.5}. Then, from (\ref{e3.7.4}) we have
	\bea\label{e3.7.6} 
	|H_{2, 1}(F_{f^{-1}}/2)|&\le& \frac{u_1(1-u_1^2)}{24\pi^2} (|C_4|+|A_4| )\sqrt{1-\frac{B_4^2}{4A_4C_4}}\nonumber\\ &\le& \frac{\eta_2(u_1)}{144\pi^4\sqrt{8\left(\pi^2+11\right)}},\eea
	
 where \beas \eta_2(u_1)=\left(\left({11-\pi^2}\right)u_1^4-2\pi^2u_1^2+4\pi^2\right)\sqrt{\frac{(8\pi^2+63)u_1^2+(16\pi^2+201)}{(u_1^2+2)}}.\eeas Differentiating $\eta_2(u_1)$ with respect to $u_1$ and we get
 \bea\label{eta_2}\eta_2^{\prime}(u_1)=\frac{u_1P_6(u_1)}{\left(u_1^2+2\right)^2\sqrt{\frac{(8\pi^2+63)u_1^2+(16\pi^2+201)}{(u_1^2+2)}}}, \eea
 where \beas P_6(u_1)&=&(-32\pi^4+100\pi^2+2772)u_1^6+(-160\pi^4-77\pi^2+13563)u_1^4\\&&+(-256\pi^4-1358\pi^2+17688)u_1^2-(1908\pi^2+128\pi^4).\eeas
 Here by simple calculation for $ 0.39298<u_1<1$ we have $P_6(u_1)<0$. Hence, from \eqref{eta_2}, we conclude that
 $\eta_2'(u_1)<0$. Therefore, $\eta_2(u_1)$ is monotonically decreasing on $(0.39298,1)$.

Figure~\ref{fig:eta38} shows that $\eta_2(u_1)$
	is strictly decreasing on $(0.39298,1)$.
	\begin{figure}[H]
		\centering
		\includegraphics[width=0.75\textwidth]{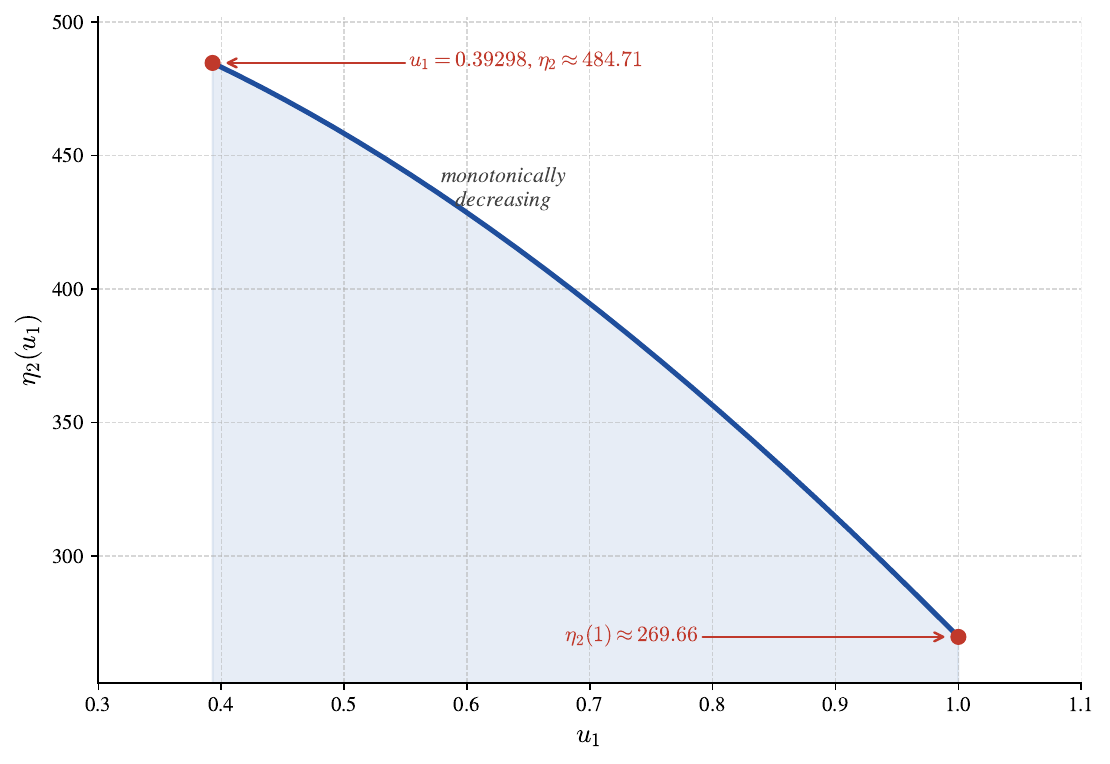}
		\caption{The function
			$\eta_2(u_1)$ attains its maximum near
			$u_1=0.39298$.}
		\label{fig:eta38}
	\end{figure}
	Therefore from \eqref{e3.7.6} we obtain
	 \beas |H_{2, 1}(F_{f^{-1}}/2)| \le \frac{\eta_2(0.39298)}{144\pi^4\sqrt{8\left(\pi^2+11\right)}}\approx0.000207.\eeas
	Hence, combining all the cases, we have \beas |H_{2, 1}(F_{f^{-1}}/2)|\le \frac{1}{36\pi^2},\eeas which is the required estimation.\\
	For sharpness, we choose the extremal function $f_2\in \mathcal{C}_{\arcsin}$ defined by \eqref{f_2}. 
	\end{proof}
	
	        \subsection{\bf Generalized Zalcman Conjecture for class $\mathcal{C}_{\arcsin}$}
	        In 1960 Zalcman's conjecture initiated the systematic study of coefficient functional in the class $\mathcal{S}$. It states that 
	        \beas |a_n^2-a_{2n-1}| \le (n-1)^2, \ \ n\ge2,\eeas with equality attained only by the Koebe function and its rotations. This conjecture immediately yields the Bieberbach inequality. Motivated by these studies in 1999 Ma \cite{Ma} formulated the generalized Zalcman conjecture 
	        \beas |a_na_m-a_{n+m-1}|\le (n-1)(m-1), \ \ m,n\ge 2. \eeas
	        Although the conjecture remains unresolved for the entire class $\mathcal{S}$, it has been verified for the subclasses
	        $\mathcal{S}^{*}$ and $S_{\mathbb{R}}$. 
	        
	        \smallskip
	        In this paper, we establish a sharp result concerning the generalized Zalcman conjecture for the initial coefficients of functions in the class
	        $\mathcal{C}_{\arcsin}$.
        \begin{theo}\label{t3.9}
			Let $f(z)=z+a_2z^2+a_3z^3+\cdots\in\mathcal{C}_{\arcsin} $. Then
		\beas
		|a_2a_3-a_4|\le \frac{1}{6\pi}.
		\eeas
		The inequality is sharp.
	\end{theo}
     \begin{proof}
         In view of (\ref{e3.3})--(\ref{e3.5}) we obtain
         \beas
         |a_2a_3-a_4|&=&\frac{1}{12\pi}\left| p_3-\left(\frac{2\pi-1}{2\pi}\right)p_1p_2+\left(\frac{7\pi^2-6\pi-12}{24\pi^2}\right)p_1^3\right|\nonumber\\&=&\frac{1}{12\pi}\left|p_3-2Bp_1p_2+Dp_1^3 \right|,\eeas
         where $B=\frac{2\pi-1}{4\pi}\approx0.42042$ and $D=\frac{7\pi^2-6\pi-12}{24\pi^2}\approx 0.16143$. Observing that $2B(2B-1)\le D\le B$, an application of Lemma \ref{l2.2} to (\ref{e3.6}) yields
         \beas |a_2a_3-a_4|&=&\frac{1}{6\pi}.\eeas
         For sharpness we take the extremal function $f_3$ defined by \eqref{f_3}.
     \end{proof}

        \subsection{Generalized Fekete--Szeg\"o functional} The Generalized Fekete--Szeg\"o functional is fundamental in geometric function theory due to its close connection with coefficient estimates, extremal problems, and geometric properties of analytic functions. In this direction in 2024 Lecko-Partyka \cite{Lecko_BSM_2024} studied the generalized Fekete--Szeg\"o functional for the class $\mathcal{S}$, defined by 
        \beas F_{\lambda,\mu}(f)&=&\left|a_3(f)-\lambda a_2^2(f)\right|-\mu |a_2(f)| \ \text{for} \  \lambda\in\mathbb{C}, \ \mu>0.\eeas
        In particular, the coefficients $a_2(f)=a_2$, $a_3(f)=a_3$, which are given in \eqref{e-1.1}, then we have 
        \bea F_{\lambda,\mu}(f)&=&\left|a_3-\lambda a_2^2\right|-\mu |a_2| \ \text{for} \  \lambda\in\mathbb{C}, \ \mu>0.\eea
        In this section, we establish the sharp upper and lower bounds for the generalized Fekete--Szego\"{o} functional associated with the class $\mathcal{C}_{\arcsin}$.        
        \begin{theo}\label{t3.10}
Let $f(z)=z+a_2z^2+a_3z^3+\cdots\in\mathcal{C}_{\arcsin}$. Then
\beas
\Theta_1\leq F_{\lambda,\mu}(f)\leq
\begin{cases}
\frac{1}{3\pi^2}\left(|2-3\lambda|-3\pi\mu\right),
& \text{if } |2-3\lambda|\ge\pi(1+3\mu),\\[2mm]
\dfrac{1}{3\pi},
& \text{if } |2-3\lambda|<\pi(1+3\mu),
\end{cases}
\eeas
where
\[
\Theta_1=
\begin{cases}
\frac{1}{3\pi^2}\left(|2-3\lambda|-3\pi\mu\right),
& \text{if } \frac{\pi}{2}(3\mu-2)\ge|2-3\lambda|,\\[3mm]
-\frac{\mu}{\pi}\sqrt{\frac{\pi}{|2-3\lambda|+\pi}},
& \text{if } \frac{\pi}{4}\left(9\mu^2-4\right)\le| 2-3\lambda|,\\[4mm]
-\frac{1}{12\pi}\left(4+\frac{9\pi\mu^2}{| 2-3\lambda|+\pi}\right),
& \text{if }
\frac{\pi}{4}\left(9\mu^2-4\right)> |2-3\lambda|>\frac{\pi}{2}(3\mu-2).
\end{cases}
\]

All the above inequalities in are sharp.
\end{theo}
\begin{proof}
Using \eqref{e3.3}--\eqref{e3.4} we obtain
\bea\label{fe1} F_{\lambda,\mu}(f)&=&\left|a_3-\lambda a_2^2\right|-\mu |a_2|\nonumber\\&=& \frac{1}{12\pi^2}\left(\left|\left(2-\pi-3\lambda\right)p_1^2+ {2\pi p_2}\right|-\left|6\pi\mu p_1\right|\right)\nonumber\\&=&\frac{1}{12\pi^2}\Phi(p_1, p_2),
\eea 
where $\Phi(p_1, p_2)=\left|Kp_1^{2}+Lp_2\right|-J|p_1|$, with $K=\left(2-\pi-3\lambda\right)$, $L=2\pi$ and $J=6\pi\mu$. We also easily conclude that $M=| 4K+2L|=4| 2-3\lambda|$.
\newline
$\bullet$\underline{\bf{Sharp upper bound:}} Let us consider $|2-3\lambda|\ge\pi(1+3\mu)$. It follows that $\left|2K+L\right|\ge|L|+J$, then by Lemma \ref{l2.6} we get
\bea\label{fe2} \Phi(p_1, p_2)\le 4\left(|2-3\lambda|-3\pi\mu\right).\eea
Putting the value of \eqref{fe2} to \eqref{fe1}, we obtain 
\beas F_{\lambda,\mu}(f)\le\frac{1}{3\pi^2}\left(|2-3\lambda|-3\pi\mu\right).\eeas
\par Equality is attained by the extremal function $f_1$ defined by \eqref{f_1}.
 
\smallskip
Similarly, if $|2-3\lambda|<\pi(1+3\mu),$
then $|2K+L|<|L|+J$.
Therefore, Lemma~\ref{l2.6} yields
\beas F_{\lambda,\mu}(f)\le\frac{1}{6\pi^2}| L|=\frac{1}{3\pi}.\eeas
For sharpness we take $f_2$ introduced by \eqref{f_2}.
Hence, we obtain the desired upper bound
\beas 
F_{\lambda,\mu}(f)\leq
\begin{cases}
\frac{1}{3\pi^2}\left(|2-3\lambda|-3\pi\mu\right),
& \text{if } |2-3\lambda|\ge\pi(1+3\mu),\\[2mm]
\dfrac{1}{3\pi},
& \text{if } |2-3\lambda|<\pi(1+3\mu)
\end{cases}
.\eeas
\newline
$\bullet$\underline{\bf{Sharp lower bound:}} To find the exact lower bound, we examine distinct scenarios.\\
{\bf A.} Let us consider $\frac{\pi}{2}(3\mu-2)\ge|2-3\lambda|$, then this implies $J\ge M+2|L|$. Then by Lemma \ref{l2.6} we obtain
\bea\label{fe3} -\Phi(p_1, p_2)\le 4\left(3\pi\mu-|2-3\lambda|\right).\eea
Now from \eqref{fe3} and \eqref{fe1} we have 
\beas F_{\lambda,\mu}(f)\ge\frac{1}{3\pi^2}\left(|2-3\lambda|-3\pi\mu\right).\eeas
For sharpness we choose $f_1$ defined by \eqref{f_1}.\\
{\bf B.} Suppose $\frac{\pi}{4}\left(9\mu^2-4\right)\le | 2-3\lambda|$. Then we have $J^2\le 2|L|\left(M+2| L|\right)$. Using Lemma \ref{l2.6} to \eqref{fe1} we obtain 
\bea\label{feb} F_{\lambda,\mu}(f)\ge-\frac{1}{6\pi^2}\left(J\sqrt{\frac{2|L|}{M+2|L|}}\right)=-\frac{\mu}{\pi}\sqrt{\frac{\pi}{|2-3\lambda|+\pi}}.\eea
To show the sharpness of inequality \eqref{feb} we consider the extremal function $f_5\in \mathcal{C}_{\operatorname{arcsin}}$ with $$p(z) = \frac{1+( \tau_1 \tau_2+\tau_1)z + \tau_2 z^2}{1+(\tau_1 \tau_2 - \tau_1)z - \tau_2 z^2},$$ where $\tau_1=\sqrt{\frac{\pi}{|2-3\lambda|+\pi}}$ and $\tau_2=\frac{(-2+3\lambda)}{|2-3\lambda|}$. Then $p_1=2\tau_1$ and $p_2=2\tau_1^2+2\left(1-\tau_1^2\right)\tau_2$, and it follows that 
\beas|Kp_1^2+Lp_2|=0. \eeas
Therefore, from \eqref{fe1}, we obtain 
\beas F_{\lambda,\mu}=-\frac{J|p_1|}{12\pi^2}=-\frac{\mu}{\pi}\sqrt{\frac{\pi}{|2-3\lambda|+\pi}}. \eeas
\\
{\bf C.} Suppose $\frac{\pi}{4}\left(9\mu^2-4\right)> | 2-3\lambda|>\frac{\pi}{2}(3\mu-2)$. Then, applying Lemma~\ref{l2.6} to \eqref{fe1}, we get 
\bea\label{fec} F_{\lambda,\mu}(f)\ge-\frac{1}{12\pi^2}\left(2| L|+\frac{J^2}{M+2|L|}\right)=-\frac{1}{12\pi}\left(4+\frac{9\pi\mu^2}{|2-3\lambda|+\pi}\right).\eea
For sharpness of the inequality \eqref{fec} we consider $f_6\in \mathcal{C}_{\arcsin}$ with $$p(z) = \frac{1+( \tau_1 \tau_2+\tau_1)z + \tau_2 z^2}{1+(\tau_1 \tau_2 - \tau_1)z - \tau_2 z^2},$$ where $\tau_1=\frac{3\pi\mu}{2\left(| 2-3\lambda|+\pi\right)}$ and $\tau_2=\frac{(-2+3\lambda)}{|2-3\lambda|}$. Then $p_1=2\tau_1$ and $p_2=2\tau_1^2+2\left(1-\tau_1^2\right)\tau_2$. It is easy to verify that 
\beas \Phi(p_1, p_2)=|Kp_1^2+Lp_2|-| Jp_1|=-\pi\left(4+\frac{9\pi\mu^2}{|2-3\lambda|+\pi}\right).\eeas
Therefore from \eqref{fe1} we get $F_{\lambda,\mu}(f)=-\frac{1}{12\pi}\left(4+\frac{9\pi\mu^2}{|2-3\lambda|+\pi}\right)$.
\par Collecting all the above cases, we obtain the desired lower bound
\[
\Theta_1=
\begin{cases}
\frac{1}{3\pi^2}\left(|2-3\lambda|-3\pi\mu\right),
& \text{if } \frac{\pi}{2}(3\mu-2)\ge|2-3\lambda|,\\[3mm]
-\frac{\mu}{\pi}\sqrt{\frac{\pi}{|2-3\lambda|+\pi}},
& \text{if } \frac{\pi}{4}\left(9\mu^2-4\right)\le | 2-3\lambda|,\\[4mm]
-\frac{1}{12\pi}\left(4+\frac{9\pi\mu^2}{| 2-3\lambda|+\pi}\right),
& \text{if }
\frac{\pi}{4}\left(9\mu^2-4\right)> | 2-3\lambda|>\frac{\pi}{2}(3\mu-2).
\end{cases}
\]
\end{proof}  
\begin{note}
  The classical Fekete--Szeg\"o functional with $\lambda=1$ coincides with the second Hankel determinant of order one, namely $|H_{2, 1}(f)=|a_3-a_2^2|$. Let $f(z)=z+a_2z^2+a_3z^3+\cdots\in\mathcal{C}_{\arcsin}$, then \beas |H_{2, 1}(f)|\le \frac{1}{3\pi}.\eeas 
   For sharpness we take $f_2$ defined by \eqref{f_2}.
\end{note}


\begin{thebibliography}{99}
	
	\bibitem{Ali_BMSS_2003} R. Ali, Coefficients of the inverse of strongly starlike functions, {\textit{Bull. Malays. Math. Sci. Soc.}}, {\bf{26}}(2003), 63--71.
	
	\bibitem{Ali_Vasudevarao_PAMS_2018} M. F. Ali and V. Allu, On logarithmic coefficients of some close-to-convex functions, {\textit{Proc. Amer. Math. Soc.}}, {\bf{146}}(3)(2018), 1131--1142.
	
	\bibitem{Ali_Allu_BAMS} M. F. Ali and V. Allu, Logarithmic coefficients of some close-to-convex functions, {\textit{Bull. Aust. Math. Soc.}}, {\bf{95}}(2017), 228-237.
	
	\bibitem{Alotabi et. al.} A. Alotaibi, M. Arif, M. A. Alghamdi and S. Hussain, Starlikeness associated with cosine hyperbolic function, {\textit{Math.}}, {\bf 8}(2020), 1118.

    \bibitem{DEA_JFT} D. Alimohammadi, E. A. Analouei, T. Bulboac$\check{a}$, and N. E. Cho, Logarithmic coefficient bounds and coefficient conjectures for classes associated with convex functions, {\textit{J. Funct. Space.}}, 2021(1)(2021), P. 6690027.
	
	\bibitem{Allu_Shaji_bams_2025} V. Allu and A. Shaji, Second Hankel determinant for logarithmic inverse coefficients of convex and starlike functions, {\textit{Bull. Aust. Math. Soc.}}, {\bf{111}}(1)(2025), 128--139.

    \bibitem{Arora_BKMS_2022} K. Arora and S. S. Kumar, Starlike functions associated with a petal shaped domain,{\textit{ Bull. Korean Math. Soc.}}, {\bf{59}}(4)(2022), 993--1010.
	
	\bibitem{ASS} A. Banerjee, S. Majumder and S. Panja, Hankel and Toeplitz determinant estimates for starlike functions via Schwarz function techniques,{\textit{ Ukr. Math. J.}}, (Accepted).
	
	\bibitem{Bano} K. Bano, M. Raza, Starlike functions associated with cosine functions, {\textit{Bull. Iran. Math. Soc.}}, {\bf{47}}(2021), 1513--1532.
	 
	\bibitem{CKKLS_	RACSAM_2020} N. E. Cho, B. Kowalczyk, O. S. Kwon, A. Lecko and Y. J. Sim, On the third logarithmic coefficient in some subclasses of close-to-convex functions, {\textit{Rev. R. Acad. Cienc. Exactas F´ıs.
	Nat. Ser. A Mat. RACSAM}}, {\bf{114}}(2)(2020), Paper No. 52, pp-14.

    \bibitem{cho} N. E. Cho, V. Kumar, S. S. Kumar, V. Ravichandran, Radius problems for starlike functions associated with the sine function, {\textit{Bull. Iran. Math. Soc.}}, {\bf{45}}(2019), 213--232.
	
    \bibitem{CKL1} N. E. Cho, B. Kowalczyk and A. Lecko, Sharp bounds of some coefficient functionals over the class of functions convex in the direction of the imaginary axis, {\textit{Bull. Aust. Math. Soc.}}, {\bf{100}}(2019), 86--96.
	
	\bibitem{CKS1} J. H. Choi, Y. C. Kim and T. Sugawa, A general approach to the Fekete-Szeg\"{o} problem, {\textit{J. Math. Soc. Japan}}, {\bf{59}}(2007), 707--727.
	
	\bibitem{PL_D_1983} P. L. Duren, Univalent Functions, Springer-Verlag, New York, 1983.
	
	\bibitem{Eker et al} S. S. Eker, A. Lecko, B. Ceki and B.Seker, The second Hankel determinant of logarithmic coefficients for strongly Ozaki close-to-convex functions,
    {\textit{Bull. Malays. Math. Sci. Soc.}}, {\bf 46}(6)(2023), 183.
	
	\bibitem{Janowski_Anpolon_1970} W. Janowski, Extremal problems for a family of functions with positive real part and for some related families, {\textit{Ann. Polon. Math.}}, {\bf{23}}(1970/71), 159--177.
	
%	\bibitem{Jantang et al} A. Janteng, S. A. Halim, and M. Darus, Hankel determinant for starlike and convex functions, {\textit{Int. J. Math. Anal.}}, {\bf 1}(13)(2007), 619–625.
	
	\bibitem{Krgar_JA_2019} R. Kargar, On logarithmic coefficients of certain starlike functions related to the vertical strip, {\textit{J. Anal.}}, {\bf{27}}(2019), 985--995.
	
	\bibitem{KL_BAMS_2022} B. Kowalczyk and A. Lecko, Second Hankel determinant of logarithmic coefficients of convex and starlike functions, {\textit{Bull. Aus. Math. Soc.}}, {\bf{105}}(3)(2022), 458--467.
	
	\bibitem{KL_RACSAM_2023} B. Kowalczyk and A. Lecko, The second Hankel determinant of the logarithmic coefficients of strongly starlike and strongly convex functions, {\textit{Rev. R. Acad. Cienc. Exactas F´ıs. Nat. Ser. A Mat. RACSAM}}, {\bf{117}}(2)(2023), Paper No. 91.
	
	\bibitem{Kumar_AMP_2021} S. S. Kumar and G. Kamaljeet, A cardioid domain and starlike functions, {\textit{Anal. Math. Phys.}}, {\bf{11}}(2)(2021), pp-34.
	
	
	\bibitem{Kumar_Yadav_IJS_2026} S. S. Kumar and P. Yadav, On a Class of Starlike Functions Associated with a Bean Shaped domain, {\textit{Iran. J. Sci.}}, {\bf{50}}(2026), 197--211.

    \bibitem{Lecko_BSM_2024} A. Lecko and D. Partyka, A generalized Fekete-Szegö functional and initial successive coefficients of univalent functions, {\textit{ Bull. Sci. Math.}}, {\bf 197}(2024), 103527.
	
	\bibitem{Lecko et al} A. Lecko and B. Smiarowska, Sharp inequalities for Zalcman functional of logarithmic coefficients of inverse functions in certain classes of analytic functions, {\textit{J. Math. Inequal.}}, {\bf19}(1)(2025), 81--97. 

	\bibitem{Lowner_annalen_1923} K. L\"{o}wner, Untersuchungen \"{u}ber schlichte konforme Abbildungen des Einheitskreises. I, {\textit{Math. Annalen}}, {\bf{89}}(1)(1923), 103--121.
	
	\bibitem{Ma} W. Ma, Generalized Zalcman conjecture for starlike and typically real functions,{\textit{ J. Math. Anal. Appl.}}, {\bf{234}}(1999), 328--339.
	
	\bibitem{Ma+Minda_1992} W. C. Ma and D. Minda, A unified treatment of some special classes of univalent functions,{\textit{Proceedings of the Conference on Complex Analysis}} (Tianjin, 1992), pp. 157–169, Conf.
	Proc. Lecture Notes Anal., Vol. I, International Press, Cambridge, MA, 1994.
	
	\bibitem{Mwndiratta et al_BMMS_2015} R. Mendiratta, S. Nagpal, and V. Ravichandran, On a subclass of strongly starlike functions associated with exponentia function, {\textit{Bull. Malays. Math. Sci. Soc.}}, {\bf{38}}(1)(2015), 365--386.
	
%	\bibitem{Milin_1971} I. M. Milin, Univalent Functions and Orthonormal Systems, {\textit{Transl. Math. Monogr.}}, Vol. {\bf{49}}, AMS, 1977.
	
	\bibitem{MK_BIMS_2023}  M. Mundalia and S. S. Kumar, Coefficient problems for certain close-to-convex functions, {\textit{Bull. Iran. Math. Soc.}} {\bf{49}}(1)(2023), Paper No. 5.
	
	\bibitem{Majumder et al._JA_2026} S. Majumder, D. Pramanik and N. Sarkar, The second Hankel determinant for logarithmic coefficients of inverse convex functions of a given order, {J. Anal.}, {\bf{34}}(2026), 365--388.
	
	\bibitem{PSW_RM} S. Ponnusamy, N. L. Sharma and K. J. Wirths, Logarithmic Coefficients of the Inverse of Univalent Functions, {\textit{Result. Math.}}, {\bf{73}}(2018), 160.
	
	\bibitem{Ponnusamy_BSM_2021} S. Ponnusamy and T. Sugawa, Sharp inequalities for logarithmic coefficients and their applications, {\textit{Bull. Sci. Math.,}} {\bf{166}}(2021), 102931.
	
	\bibitem{Ponnusamy_JAMS_2020} S. Ponnusamy, N. L. Sharma, and K. J. Wirths, Logarithmic coefficients problems in families related to starlike and convex functions, {\textit{J. Aust. Math. Soc.}}, {\bf{109}}(2020), 230--249.
	
	\bibitem{Pom_1967} C. Pommerenke, On the Hankel determinants of univalent functions, {\textit{Math.}}, {\bf 14}(1)(1976), 108--112.
	
%	\bibitem{Raza et. al} M. Raza, A. Riaz, D. K. Thomas and P. Zaprawa, Third Hankel determinant for starlike and convex functions associated with the exponential function, {\textit{Bol. Soc. Mat. Mex.}}, {\bf{31}}(1)(2025), 16.
	
	\bibitem{Ronning_AUMC_1991} F. Rønning, On starlike functions associated with parabolic regions, {\textit{Ann. Univ. Mariae Curie-Sk lodowska Sect.}},  {\bf{45}}(1991), 117--122.
	
	\bibitem{Sabir_Ali_2025} P. O. Sabir and A. A. Ali, Toeplitz and Hankel determinants of logarithmic coefficients for $r$-valent $q$-starlike and $r$-valent $q$-convex functions, {\textit{MethodsX}}, {\bf{15}}(2025), 103463.
	
	\bibitem{Sim_Thomas_Symmetry_2020} Y. J. Sim and D. K. Thomas, On the difference of inverse coefficients of univalent functions, {\textit{Symmetry}}, {\bf{12}}(12)(2020).
	
	\bibitem{Thomas_PAMS_2016} D. Thomas, On the logarithmic coefficients of close-to-convex functions, {\textit{Proc. Amer. Math. Soc.}}, {\bf{144}}(4)(2016), 1681--1687.
	
	\bibitem{Zaprawa_BSMM_2021} P. Zaprawa, Initial logarithmic coefficients for functions starlike with respect to symmetric points, {\textit{Bol. Soc. Mat. Mex.}}, {\bf{27}}(3)(2021), Paper No. 62, pp-13.
	
	
	
	
	
	
	
	
	
	
	
	

	

	

	

	
	
	
	
	
	
	
	
	
	
	
\end{thebibliography}
\end{document}